\documentclass[10pt]{article}

\usepackage[a4paper, margin=1.4in]{geometry}

\usepackage{amsmath} 
\usepackage{amssymb, amsthm} 
\usepackage[usenames, dvipsnames]{color}

\usepackage{graphicx}

         \newtheorem{theorem}{Theorem}%[section]
	 \newtheorem{proposition}[theorem]{Proposition}%[section] 
	 \newtheorem{lemma}[theorem]{Lemma}
         \newtheorem{reform}{Problem}

\theoremstyle{definition}        
         \newtheorem{remark}[theorem]{Remark}%[section]  
\newcommand \Nbb {\mathbb{N}}
\newcommand \Rbb {\mathbb{R}}
\newcommand \Tbb {\mathbb{T}}
\newcommand \Cbb {\mathbb{C}}
\newcommand \Dbb {\mathbb{D}}
\newcommand \Zbb {\mathbb{Z}}

\newcommand \be {\begin{equation}}
\newcommand \ee {\end{equation}}

\def \h {\hat}
\def \t {\tilde} 

\newcommand \pa {\partial}

\newcommand \al {\alpha}
\newcommand \de {\delta}
\newcommand \ep {\epsilon}
\newcommand \varep {\varepsilon}
\newcommand \ga {\gamma}
\newcommand \Ga {\Gamma}

\newcommand \si {\sigma}
\newcommand \ka {\kappa}
\newcommand \om {\omega}%{\boldsymbol\omega}
\newcommand \Om {\Omega}%{\boldsymbol\omega}

\newcommand \Ccal {\mathcal{C}}
 
\newcommand \Fcal {\mathcal{F}} 
\newcommand \Lcal {\mathcal{L}}
\newcommand \Scal {\mathcal{S}}
\newcommand \Ncal {\mathcal{N}}
\newcommand \Ical {\mathcal{I}}
\newcommand \Wcal {\mathcal{W}}
\newcommand \Vcal {\mathcal{V}}

\newcommand \dist {\textup{dist}}

\DeclareMathOperator\Imag{Im}

\title{Self-intersecting interfaces for stationary solutions\\ of the
  two-fluid Euler equations}
\author{Diego C\'ordoba\thanks{dcg@icmat.es}, Alberto Enciso\thanks{aenciso@icmat.es} and Nastasia Grubic\thanks{nastasia.grubic@icmat.es}}
\date{\normalsize Instituto de Ciencias Matem\'aticas\\ Consejo Superior de
  Investigaciones Cient\'\i ficas\\ 28049 Madrid, Spain}

\begin{document}

\maketitle

\begin{abstract}
  We prove that there are stationary solutions to the 2D
  incompressible free boundary Euler equations with two fluids, possibly with a small gravity constant, that feature a splash singularity. More
  precisely, in the solutions we construct the interface is a
  $\Ccal^{2,\al}$~smooth curve that intersects itself at one point,
  and the vorticity density on the interface is of
  class~$\Ccal^\al$. The proof consists in perturbing Crapper's family
  of formal stationary solutions with one fluid, so the crux is to
  introduce a small but positive second-fluid density. To do so, we
  use a novel set of weighted estimates for self-intersecting
  interfaces that squeeze an incompressible fluid. These estimates
  will also be applied to interface evolution problems in a
  forthcoming paper.
\end{abstract}

\section{Introduction}

Let us consider the two-fluid
	incompressible irrotational Euler equations in $\Rbb^2$, where
        a time-dependent interface
$$
\Ga(t) = \{z(\al,t) = (z_1(\al,t), z_2(\al,t)) \,|\, \alpha\in \Rbb\}
$$
	separates the plane in two disjoint open regions $\Omega_j(t)$, with
        $j=1,2$. Each $\Omega_j(t)$ denotes the region occupied by the
        two different fluids with velocities $v^j=(v_1^j,v_2^j)$,
        different constant densities $\rho_j$ and pressures $p^j$,
        which evolve according to the Euler equations:
\begin{subequations}\label{equations}
\begin{align}
\rho_j (\pa_t v^j + (v^j \cdot \nabla )v^j = - \nabla p^j - g \rho_j \, e_2 \quad &\text{in} \quad \Omega_j(t),\\
\nabla\cdot v^j=0  \quad \mathrm{and} \quad \nabla^{\perp}v^j=0  \quad &\text{in} \quad \Omega_j(t), \\
(\pa_t z- v^j) \cdot (\pa_\al z)^\perp = 0 \quad &\mathrm{on} \quad \Ga(t),\\
p^1 - p^2 = -\sigma K \quad &\mathrm{on} \quad \Ga(t).
\end{align} 
\end{subequations}
Here $j\in\{1,2\}$, $\sigma>0$ is the surface tension coefficient,
$e_2$ is the second vector of a Cartesian basis and
$K$ is the curvature of the interface. 
	
In the case of water waves (that is, $\rho_1 =\si=0$
and the fluid is irrotational),
Castro, C\'ordoba, Fefferman, Gancedo and G\'omez-Serrano \cite{CCFGG1}
proved the formation of splash singularities in finite time, meaning
that the interface remains smooth but self-intersects at a point. The
system of water waves in $\Rbb^2$ can be written in terms of the
boundary curve $z(\al,t) $ and the vorticity density in the boundary,
defined through the formula
$$
\nabla^\perp \cdot v  =:  \om(\al,t)\, \delta (x - z(\al,t))\,.
$$
Notice that for these types of singularities to happen the amplitude
of the  vorticity $\om(\al,t)$ has to blow-up at the point of the
self-intersection.   With a different approach Coutand and Shkoller
proved  in \cite{CS2} that these singularities also occur when you
consider non-zero vorticity. It is also known that the presence of surface tension does not prevent the formation of a splash singularity~\cite{CCFGG2}. 

A remarkable result of Fefferman, Ionescu and Lie~\cite{FIL} (see also \cite{CS1}) ensures
that, if both fluid densities are strictly positive, splash
singularities cannot occur. More precisely, if the interface does
not have any self-intersections at time~0, and the curve~$z(\alpha,t)$
and the velocity field of the fluid on the interface remain bounded
in~$\Ccal^4$ and~$\Ccal^3$, respectively, up to time~$T$, their theorem
ensures that the interface does not develop any self-intersections in
the time interval~$[0,T]$. Furthermore, the proof shows that their
result on the absence of splash singularities with two fluids still
holds if one replaces the above regularity hypotheses by
weaker regularity requirements.
% that the vorticity on the boundary and the curve be
%bounded in the H\"older spaces~$\Ccal^\alpha$ and~$\Ccal^{2,\alpha}$,
%respectively, up to time~$T$.

Our mid-term objective is to construct a scenario of
singularity formation for the Euler equations with two fluids, which
must involve interfaces or vorticities of low regularity
by the aforementioned result of Fefferman, Ionescu and
Lie. Intuitively, if the interface looses smoothness as one approaches
the time of self-intersection because the two collapsing waves become
sharper and sharper, this would allow the incompressible fluid in
between to escape more easily, and the collapsing waves would still be able to travel towards each other and intersect in finite time. The self-intersection of an interface has been observed numerically for the vortex sheet dynamics in \cite{HLS} where the strength of the vortex sheet and the interfacial curvature diverge at the collision time.

As a first step in this program, our goal in this paper is to prove
the existence of {\em stationary}\/ splash singularities with two
fluids. More precisely, for any value of the density of the lower
fluid~$\rho_2$ and all small enough upper fluid density~$\rho_1$ we
show that there are stationary solutions that exhibit a splash. The a
priori regularity of the vorticity and of the interface is slightly
below the requirements of the no-splash theorem of Fefferman, Ionescu
and Lie. Omitting the periodicity and decay conditions on the
velocities at infinity, which will be made precise in
Section~\ref{S.setting}, our main result can be stated as follows:

\begin{theorem}\label{T.main}
Let us fix the density of the second fluid $\rho_2>0$. Then there exists  some surface tension coefficient~$\si>0$, such that for any
small enough upper fluid density~$\rho_1>0$ and~$g$ there is a
  stationary solution to the two-fluid Euler equations for
  which the interface~$\Ga$ has a splash singularity. The regularity
  of~$\Ga$ and of the vorticity on the interface is~$\Ccal^2$
  and~$\Ccal^\alpha$, respectively.
\end{theorem}

It is worth emphasizing that, although recently there has been much
interest in stationary solutions of the Euler equations~\cite{AAW,BDT,Chae,Sverak,CS,CEG,boeck,EP12,EP15,EP17,Nadirashvili3,Nadirashvili,Nadirashvili2},
our motivation for this paper comes mainly from the dynamical
case. To see the relevance of this result on stationary solutions for our
program to establish the formation of splash singularities with two
fluids, it is convenient to compare the proof of Theorem~\ref{T.main}
with our previous paper~\cite{CEG}. There we started from a family of
(sometimes formal) one-parameter family of exact stationary solutions
to the water waves equations without gravity (i.e., $\rho_1=0$
and~$g=0$) known as Crapper waves~\cite{crapper}, which exhibit a
splash singularity for a certain value of the parameter. 
\begin{figure}[h]
	\centering
	\includegraphics[scale=0.15]{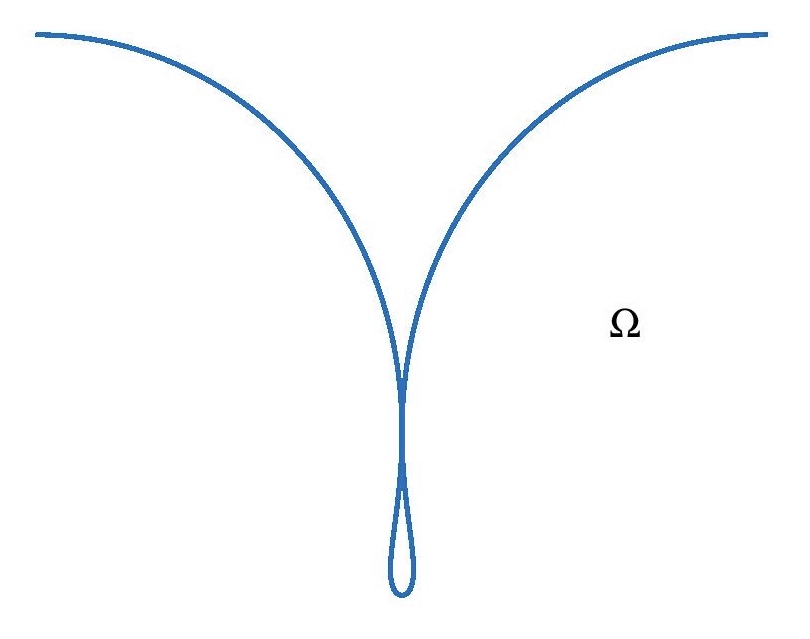}
	\caption{The critical Crapper wave}
\end{figure}
The interface and the vorticity on the boundary are smooth for all admissible values of the parameter. We then perturbed the Crapper waves, using an
implicit function theorem, to construct a one-parameter family of
smooth stationary solutions for all small enough values of~$g$
and~$\rho_1\geq0$, showing that in particular there are stationary splash
singularities if~$\rho_1=0$. For positive upper fluid densities, we
managed to prove that there are stationary splash singularities
featuring smooth ``almost-splash'' interfaces, meaning curves that are
arbitrarily close to self-intersecting, but we could not show that
stationary splash singularities can arise with two
fluids because the key estimates needed to apply the implicit
function theorem break down in this case. In plain words, what happens
is that when the splash occurs in the one-fluid case it suffices to study the
equations inside the inner-regular domain occupied by the lower fluid,
while in the two-fluid case one must also study what happens in the
region occupied by the upper fluid, which has a cusp. The strategy of perturbing Crapper waves to obtain a larger class of stationary water waves was introduced in \cite{AAW} and exploited in \cite{boeck}.

Hence, the heart of this paper is to develop a new set of estimates
that permits us to analyze two fluids separated by a splash curve. The
functional framework in which we managed to do this, which owes much to
the work of Maz'ya and Soloviev on singular integral operators in
domains with cusps~\cite{MazSo}, is that of Sobolev spaces with
singular weights. Essentially, the singularity of the weight, whose
strength depends on a parameter~$\mu$,  is
located at the splash point of the interface and accounts for the fact
that the solutions we construct (and the very interface curve) is not
smooth, in contrast to what happens in our previous
paper~\cite{CEG}. Loosely speaking, the role of the weighted estimates
that underlie the proof of Theorem~\ref{T.main} is to describe what
happens when an incompressible fluid get pinched by a cusp. We will
exploit this idea in the dynamical situation elsewhere~\cite{forthcoming}, so for future
reference we have chosen to carry out the weighted estimates in
slightly more generality than we need here. Thus we shall see
that the specific $\Ccal^\al$--$\Ccal^{2,\al}$ regularity we find here is by
no means a fundamental feature of the functional spaces that we
employ here, but a consequence under this functional framework of the
smoothness and specific geometric structure of the Crapper waves that
we are perturbing. In particular, lower vorticity and
interface regularity may arise from the same functional framework with
a parameter~$\mu$ other than the one corresponding to perturbations
of the Crapper wave, namely $\mu=1$.

The paper is organized as follows. In Section~\ref{S.setting} we
carefully present the setting of the problem and introduce some
notation. The basic weighted estimates for singular integral operators
that we employ in this paper are derived in Section~\ref{S.estimates},
where we also introduce the function spaces that we work with. In
Section~\ref{S.invertibility} we discuss the invertibility on splash domains of
an operator associated with the Neumann problem for the
Laplacian. Finally, in
Section~\ref{S.IFT} we present the proof of Theorem~\ref{T.main},
which follows from an implicit function theorem argument that hinges
on the estimates derived in the preceding sections. The statements and
(sketches of) proofs of some auxiliary
results that are adaptations and small refinements of estimates
available in the literature have been relegated to
Section~\ref{S.aux}. The paper concludes with an Appendix where we
recall some estimates for the Hardy operator, some Fourier multiplier
theorems on weighted Sobolev spaces and some results on harmonic
functions that are employed throughout the paper.

\section{Setting of the problem}\label{S.setting}

In this paper, we restrict attention to stationary periodic solutions of the system \eqref{equations}. The parametrization of the interface is thus time-independent and satisfies 
\begin{equation}\label{eq:periodicity}
z(\alpha + 2\pi) = z(\alpha) + 2\pi.
\end{equation}
Periodicity implies we may restrict attention to one period of $z$, we therefore usually work with
$$
\Ga :=  \{z(\al) \,|\, \alpha\in (-\pi, \pi)\}
$$ 
where we may assume $\Ga\subseteq \Scal_\pi := \{z\in\Cbb  : |\Re z| < \pi  \}$. We further assume $\Ga$ is symmetric with respect to the $y$-axis, i.e.
\begin{equation}\label{eq:sym}
z(-\alpha)= -z(\alpha)^*
\end{equation}
where $^*$ denotes complex conjugation. Finally, we assume there exists a point $z_*\in \Scal_\pi$ such that 
\be\label{eq:splash}
\exists\al_*\in (0, \pi): \qquad z_*=z( \pm\al_*), \qquad \theta(\pm\al_*) = \pm\pi/2
\ee   
but that $\Ga$ is otherwise a chord-arc curve, that is 
\be\label{arcchord}
\Fcal_{\varepsilon}(\Ga) := \sup_{(\al, \beta)\in (B_\varepsilon(\al_*)\times B_\varepsilon(-\al_*))^c}\Fcal(z)(\al, \beta) < \infty
\ee
for any small $\varepsilon>0$, where $\theta$ is the tangent angle and
\be\label{eq:defFcal}
\Fcal(z)(\al, \beta) :=\begin{cases}
                        \frac{|e^{i\al} - e^{i\beta}|}{|z(\al)-z(\beta)|} \quad &\al\neq \beta\\
                        \frac{1}{|z_\al(\al)|}                       \quad &\al = \beta.
                       \end{cases}  
\ee
In particular, $\Ga$ self-intersects and $\Scal_\pi \setminus \Ga$ is a union of three disjoint open sets. The region below $\Ga$ (i.e. the region containing $(-\infty, a]i$ for some $a\in \Rbb$) is connected and we denote it by $\Om$. Its complement $\Om^c:= \Scal_\pi \setminus \overline{\Om}$ is a union of two connected disjoint open sets $\Om_i$ ($i=1,2$) both with an outward cusp and a common tip at $z_*$. More precisely, we have
$$
\Scal_\pi \setminus \overline{\Om} = \Om_1\cup \Om_2, \qquad \overline{\Om}_1\cap \overline{\Om}_2 = \{z_*\}
$$
where we assume $\Om_1$ is bounded and $\Om_2$ unbounded.

Under these conditions, the Euler equations \label{equations} for the $2\pi$-periodic velocity $v$ and the corresponding $2\pi$-periodic pressure $p$ simplify to
\begin{subequations}\label{equations}
\begin{align}
\rho (v \cdot \nabla )v = - \nabla p - g \rho \, e_2 \qquad &\text{in } \ \Scal_\pi\setminus\Ga ,\\
\nabla\cdot v=0  \quad \mathrm{and} \quad \nabla^{\perp}v=0  \qquad &\text{in } \ \Scal_\pi\setminus\Ga, \label{eq:subeq2} \\
v^j \cdot (\pa_\al z)^\perp = 0 \qquad &\mathrm{on } \, \ \Ga, \label{eq:subeq3}\\
p^1 - p^2 = -\sigma K \qquad &\mathrm{on }\, \ \Ga,
\end{align} 
\end{subequations}
where $\si>0$ is the surface tension coefficient, $K$ is the curvature of the interface
$$
K(z) = \frac{z_{\al\al}\cdot z_\al^\perp}{|z_\al|^3},
$$
and we have used superscript $j$ to denote restrictions of $p$ and $v$ to the respective domains $\Om^c$ and $\Om$. Moreover, each fluid is assumed to have constant density, i.e. we have $\rho(z) \equiv \rho_2>0$ when $z\in\Om$, respectively $\rho(z)\equiv \rho_1\geq 0$ for $z\in \Scal_\pi\setminus \Om$.
\iffalse
$$
\rho(z) = \begin{cases}
        \rho_1, \qquad z\in \Scal_\pi\setminus \Om,
        \\
        \rho_2, \qquad z\in \overline{\Omega}
       \end{cases}
$$
where $\rho_2>0$ and $\rho_1\geq 0$. 
\fi
It remains to specify the behavior of the velocity far from the interface; we impose
\be\label{vinf}
\aligned
v^1 \rightarrow 0 \quad &\mathrm{as} \quad y\rightarrow \infty, \\
v^2 \rightarrow 1 \quad &\mathrm{as} \quad y \rightarrow -\infty,
\endaligned
\ee
uniformly in $x$.

We now briefly recall how to rewrite the system of equations \eqref{equations}-\eqref{vinf} in terms of vorticity $\om := \nabla^\perp \cdot v$ and  parametrization of the interface, for more details cf. \cite{CEG}. The vorticity $\om$ is assumed to be a measure supported on the interface, i.e.
$$
\om (z) = \om(\al)\delta (z - z(\al)),
$$
where we slightly abuse the notation and let $\om$ denote the amplitude of vorticity along $\Ga$ as well. Away from the interface the velocity of the fluid is essentially given by the periodic Birkhoff--Rott integral $BR(z,\om)$; its complex conjugate is given by
$$
BR(z, \om)^* = \frac{1}{4\pi i}\int_{-\pi}^\pi \, \om(\al')\cot\Big(\frac{z-z(\al')}{2}\Big)\,d\al'.
$$
Far from the interface, we have  
$$
BR(z, \om)^* \longrightarrow \mp \frac{1}{4\pi}\int_{-\pi}^\pi \, \om(\al')  d\al' \qquad \Im z \gtrless 0, \quad  |z|\rightarrow \infty,
$$
uniformly, hence in order to satisfy \eqref{vinf}, we require $\int_{-\pi}^\pi \, \om(\al') d\al' = 2\pi$ and add a suitable constant velocity field to the Birkhoff-Rott integral, i.e. away from the interface we set
$$
v(z) = \frac{1}{2} + BR(z, \om).
$$
\iffalse
In connection with this, recall also the following formula for the multiplication of complex numbers
$$
\Re(wz^*) = z\cdot w, \quad \Im(wz^*)  = -z\cdot w^\perp.
$$
\fi
On the other hand, approaching any point on $\Ga\setminus\{z_*\}$ from inside of $\Om$ (respectively $\Om^c$) in the normal direction, we obtain   
$$
v^2 = \frac{\om}{2z_\al^*} + \frac{1}{2} + BR(z,\om), \qquad v^1 =  - \frac{\om}{2z_\al^*} + \frac{1}{2} + BR(z,\om),
$$
where the integral is now understood in the sense of principal value. In particular, the amplitude of the vorticity measures the jump in the tangential component of the velocity along the interface. 

Following \cite{OS}, we use the hodograph transform with respect to the lower fluid in order to transform the free boundary problem into a problem on a fixed domain. In fact, equations \eqref{eq:subeq2} imply we can write $v$ in terms of the flow potential $\phi$ or the stream function $\psi$ via 
$$
v = \nabla \phi = \nabla^\perp \psi.
$$
By definition, $\phi + i\psi$ is analytic on $\Om$ and can be shown to be a conformal bijection onto $\Scal_\pi \cap \Cbb_-$ extending as a homeomorphism up to the boundary, cf. \cite{CEG}. In particular, we can use
$$
\phi + i\psi:\Om\longrightarrow \Scal_\pi \cap \Cbb_-
$$
as an independent variable in place of $z = x+iy$. The parametrization is fixed by requiring 
$$
\phi(z(\al)) = \al, \quad \psi(z(\al)) = 0,
$$
which implies a simple relation between the velocity of the lower fluid and the tangent vector on the interface
$$
v^2\cdot z_\al = \nabla \phi\cdot z_\al = 1 \quad  \Rightarrow \quad 2BR(z,\om)\cdot \pa_\al z + \om = 2 - z_{1\al}.
$$
The problem can be further simplified by writing the velocity vector in terms of polar coordinates 
$$
\pa_\al z = |\pa_\al z|e^{i\theta} = e^{if},\qquad f= \theta + i \tau : \Scal_\pi \cap \Cbb_- \rightarrow \Cbb,
$$
where $f$ is analytic and continuous up to the boundary. Since $\theta$ and $\tau$ are $2\pi$-periodic conjugate
functions, they must be related by the periodic Hilbert transform, i.e. $\tau = H\theta$ and we can take $\theta$ to be our main unknown. The parametrization $z$ is then a function of $\theta$ via the integral operator 
\be
\label{zoftheta}
z(\alpha) = I(\theta)(\al) := -\pi + \int_{-\pi}^\al e^{-H\theta(\al ') + i \theta (\al')} d\al'.
\ee
As shown in \cite{CEG}, the system \eqref{equations}-\eqref{vinf} is then equivalent to the following problem: 
\begin{reform}\label{ProbA}
Find $2\pi$-periodic functions $\theta(\al)$ and $\om(\al)$, such that
$\theta$ is odd, $\om$ is even and they satisfy 
\begin{subequations}
\label{eq:BR1}
\begin{align}
 q \Big(1 + \frac{\epsilon}{2}\Big)\frac{d\theta}{d\al} + \sinh
   H\theta + g e^{-H\theta}\, \Imag I(\theta) - \ka e^{-H\theta} - \frac{\epsilon}{4}e^{H\theta}\om (\om - 2) &= 0, \label{eq:bernoulli}\\
2BR(z,\om)\cdot \pa_\al z + \om &= 2 - z_{1\al},  \label{eq:vorticity}
\end{align}
\end{subequations}
where $z := I(\theta)$ is defined by \eqref{zoftheta}. 
\end{reform}
The above problem depends on four constants $q, \ka, \ep$ and $g$, where $g$ represents the gravity, $\ka$ is the integration constant of the Bernoulli equation, $q$ is related to the surface tension coefficient via $q:= \frac{\si}{\rho_2}$, while 
$$
\ep := \frac{2\rho_1}{\rho_2-\rho_1}
$$ 
detects the presence of the upper fluid. Setting $\epsilon$ to zero, the equations decouple and we recover the capillary-gravity wave
problem as studied in \cite{AAW}. If, in
addition, we set $g = 0$, we recover the pure capillary waves problem
as formulated by Levi-Civita (see e.g.\ \cite{OS}), namely,

\begin{reform}\label{ProbC}
Find a $2\pi$-periodic, analytic function  $f = \theta + i \tau$ on the lower half-plane that satisfies 
\be
\label{ref:2}
q \frac{d\theta}{d\al} =  -\sinh H\theta
\ee
on the boundary and tends to zero at infinity. 
\end{reform}

This problem admits a family of exact solutions depending on the parameter $q$. More precisely, the family of analytic functions    
\be
\label{ref:3}
f_A(w) := 2i \log \frac{1 + Ae^{-iw}}{1 - Ae^{-iw}} 
\ee
has all the required properties, with parameter $A$ depending on $q$ via 
$$
\quad q = \frac{1 + A^2}{1-A^2},
$$
cf. Crapper \cite{crapper}. It suffices to consider $A\geq 0$, since the transformation $A
\mapsto -A$  corresponds to a translation $\al \rightarrow \al +
\pi$. The corresponding wave profiles are given by  
\be\label{eq:crapper}
z_A(\alpha) = \alpha + \frac{4i}{1 + Ae^{-i\al}} - 4i.
\ee
where the constant of integration has been chosen to ensure $z_A (\al)=\al$ for $A=0$. For sufficiently large values of parameter $A$ these solutions
can no longer be represented as a graph of a function, and eventually
self-intersect. It is not hard to see that the curve $z_A$ does not have self-intersections if and only if
 $$
 A< A_0 \approx 0.45467.
 $$
For $A=A_0$, the curve $z_A(\al)$ exhibits a splash-singularity, while for $A$
slightly larger than $A_0$ the curve intersects at exactly two points,
and the intersection is transverse.

\begin{section}{Estimates on the singular integral operators}\label{S.estimates}

In this section, we show boundedness of singular integral operators on the appropriate weighted Sobolev spaces. For convenience, we study these integrals as defined over a closed curve. More precisely, we temporarily assume the interface is a closed bounded curve, e.g. we consider the image of $\Ga$ under the exponential map $P(z) = -e^{-iz}$, which, to simplify notation, we continue denoting $\Ga$. After a translation, we may assume the singular point is at the origin, i.e. $z_* = 0$. 

As for the regularity, we assume $\Ga$ is of class $\Ccal^{2,\lambda}$ everywhere except possibly at the singular point, where it is at least $\Ccal^{1,\lambda}$ for some possibly different $\lambda\in (0,1)$. The assumptions \eqref{eq:splash}--\eqref{arcchord} then imply there exists a neighborhood $B(z_*)$ of the singular point such that $B(z_*)\cap \Ga$ consists of exactly two connected components $\Ga^\pm$ which can be parametrized as a graph. More precisely, we have $\Ga^+ \cup \Ga^- = B(z_*)\cap \Ga$ with
$$
\Ga^\pm =\big\{ x \pm i\ka(x) \,:\, |x|<\de \big\},
$$ 
where $ \ka'(0) = \ka(0) = 0$ and $\ka''$ is H\"{o}lder continuous on $[-\de,\de]\setminus\{0\}$. Near the singular point, we assume 
\be\label{reg_assum}
|x|^{1-\mu}\ka''(x)\in H^1([-\de, \de]), \ \quad \ \lim_{x\rightarrow 0} |x|^{1-\mu}\ka''(x) = k>0
\ee
for some $\mu \in (0,1]$. In particular, the above implies  
\be\label{reg_implication}
\mu |x|^{-\mu}\ka'(x) \rightarrow k, \qquad \mu(\mu+1)|x|^{-(1+\mu)}\ka(x) \rightarrow k,
\ee
By possibly making $\de$ smaller, we may assume that $\rho(x):=2\ka(x)$ is strictly monotonically increasing on $[0,\de]$ respectively decreasing on $[-\de, 0]$. Here, we will only need $\mu = 1$, however, for later reference, we prove some results for general $\mu$.

We finish with some notation. In general, we append subscript $\pm$ to a point $z\in\Ga$ whenever it's an element of $\Ga^\pm$ in the graph parametrization, e.g. 
$$
z_\pm = x \pm i\ka(x) \in\Ga^\pm.
$$ 
Moreover, for fixed $0<x<\de$ we split $[-\de, \de]$ into three parts
$$
I_l(x):=[-x,(1-\epsilon)x], \quad I_c(x):=[(1-\epsilon)x, (1+\epsilon)x], \quad I_r(x):=[-\de, -x]\cup [(1+\epsilon)x, \de]
$$ 
for some small $\epsilon>0$ and denote the corresponding part of $\Ga^\pm$ by $\Ga^\pm_l(x)$, $\Ga^\pm_c(x)$ and $\Ga^\pm_r(x)$ respectively. More precisely, we have
$$
\Ga^\pm_{l, c, r}(x) := \{ \, q \in \Ga^\pm \, : \,  u \in I_{l, c, r}(x)  \}.
$$
Finally, given two (positive) quantities $f,g$, we define
$$
f \lesssim g :\Leftrightarrow \quad \exists c>0: \ f \leq c g
$$
and similarly
$$
f \sim g\, :\Leftrightarrow \quad \exists c>0: \ c^{-1} g \leq f \leq c g  .
$$

\begin{subsection}{Function spaces}

For $\beta\in \Rbb$, we define the weighted Lebesgue space with respect to the power weight $|q|^\beta$ as 
$$
\Lcal_{p_,\beta}(\Ga) := \{\phi: \Ga \rightarrow \Rbb \,\,|\,\, |q|^\beta |\phi(q)| \in L^p(\Ga) \},
$$
endowed with the norm
$$
\|\phi\|_{p,\beta}^p = \int_\Ga |q|^{\beta p}|\phi(q)|^p ds_q.
$$
If  $0 < \beta + p^{-1}<1$, then $|q|^\beta$ is a M\"{u}ckenhaupt weight and $\Lcal_{p_,\beta}(\Ga) \subseteq L^1(\Ga)$. We will also need the weighted Sobolev space  
\iffalse
Generally, we have  
$$
\beta + p^{-1}<{q}^{-1} \quad \Longrightarrow \quad \Lcal_{p_,\beta}(\Ga) \subseteq L^{q}(\Ga).
$$ 
We will make use of two families of weighted Sobolev spaces 
\fi
$$
\phi \in \Wcal^{1,p}_\beta(\Ga) \ \ :\Longleftrightarrow \ \ \phi,\, \frac{d\phi}{ds} \in \Lcal_{p_,\beta}(\Ga)
$$
\iffalse
and 
$$
\phi \in \Lcal^1_{p,\beta}(\Ga) \ \ :\Longleftrightarrow \ \ \phi\in\Lcal_{p,\beta-1}(\Ga) , \ \frac{d\phi}{ds} \in \Lcal_{p_,\beta}(\Ga),
$$
\fi
where $d/ds$ denotes the derivative with respect to the arc-length parametrization of $\Ga$. If $0 < \beta + p^{-1}<1$, then any $\phi \in \Wcal^{1,p}_\beta(\Ga)$ is necessarily continuous and $\phi - \phi(0) \in \Lcal_{p, \beta - 1}(\Ga)$.
\iffalse
$$
\Wcal^{1,p}_\beta(\Ga) = \Lcal^1_{p,\beta}(\Ga) \oplus \Rbb.
$$
Finally, we define
$$
\phi\in\Ncal^{1,-}_{p,\beta}(\Ga) \ \ :\Longleftrightarrow \ \ \phi \in \Wcal^{1,p}_{\beta}(\Ga) \, , \ \phi(z_+) - \phi(z_-) \in \Lcal_{p,\beta -\mu}(\Ga^+\cup\Ga^-)
$$
with the norm 
$$
\|\phi\|^p_{\Ncal^{1,-}_{p,\beta}} := \|\phi\|^p_{\Wcal^{1,p}_{\beta}} + \int_{\Ga^+\cup\Ga^-} |z|^{(\beta-\mu) p}|\phi(z_+) - \phi(z_-)|^p\,ds_z.
$$
\fi
For $\lambda \in \Rbb$, we define $\Ccal_\lambda(\Ga)$, the space of weighted continuous functions, to be 
$$
\Ccal_\lambda(\Ga) := \{\phi: \Ga \rightarrow \Rbb \,\,|\,\, q^\lambda \phi(q) \in \Ccal^0\big(\Ga\big) \}
$$
endowed with the norm
$$
\|\phi\|_\lambda := \sup_{q\in\Ga} \, |q|^\lambda |\phi(q)|.
$$
Generalization to $k$-times differentiable functions is straightforward; e.g. the weighted space $\Ccal^2_{\lambda}(\Ga)$ of two times continuously differentiable functions is defined to be 
$$
\phi \in \Ccal^2_{\lambda}(\Ga) \ \ :\Longleftrightarrow \ \ \frac{d^2\!\phi}{ds^2}\in \Ccal_{\lambda}(\Ga),\ \frac{d\phi}{ds}\in \Ccal_{\lambda-1}(\Ga),\ \phi\in \Ccal_{\lambda-2}(\Ga).
$$
Note, the assumptions \eqref{reg_assum} imply $\ka \in \Ccal^2_{1-\mu}([-\de,\de])$. 
\end{subsection}
\begin{subsection}{Regularity properties of the Birkhoff-Rott integral}

Here, we prove boundedness of singular integral operators on appropriate weighted Sobolev spaces and give some compactness results. To simplify notation, we use symmetry assumptions \eqref{eq:sym} for the interface and assume $\om$ is either even or odd w.r.t. $x$-axis, however everything works in exactly the same way without any symmetry assumptions. 

\begin{proposition}\label{prop:BR}
Let $1-\mu < \beta + p^{-1}< \mu$. The complex-valued linear operator 
\begin{equation}\label{def:BR}
%\begin{align}
BR(z, \cdot)^*: \Wcal^{1,p}_{\beta}(\Ga)\longrightarrow \Wcal^{1,p}_{\beta}(\Ga) 
%\end{align}
\end{equation} 
is bounded. The same is true for $BR(z, \cdot)^*z'$.
\end{proposition}

\begin{proof} For convenience, we will work with the standard form of the Birkhoff-Rott integral, which in complex notation with respect to e.g. the arc-length parametrization reads  
$$
 BR(\om,z)^*z' = \frac{1}{2\pi i}\, p.v.\int_\Ga \om(q)\bigg(\frac{z'}{z-q}\bigg)\,ds_q,
$$
where $^*$ means complex conjugation. This is justified since the difference of periodic and line kernels is bounded and  can easily be estimated in $\Lcal_{2,\beta}(\Ga)$ using Hardy inequalities (we have $\cot(w) = 1/w + O(w)$ in the neighborhood of the origin). 

We first show $BR(\om,z)^*z'\in \Lcal_{p,\beta}(\Ga)$, when $\om \in \Lcal_{p,\beta}(\Ga)$. Since the interface is a chord-arc curve outside of any neighborhood of the splash point (cf. \eqref{arcchord}), we may assume $z = z_+ \in \Ga^+$ with positive real part strictly smaller than say $\de/2$. In the graph parametrization, this gives $z_+ = x+i\ka(x)$ with $0<x<\de/2$. Moreover, it is enough to consider $q= q_\pm\in \Ga^\pm$, i.e. $q_\pm = u \pm i\ka(u)$ with $|u|\leq \de$. 

Let first $ q\in \Ga^\pm_l(x)$ or $q\in \Ga^\pm_r(x)$. Then we have the estimate  
\be\label{est:lrbasic}
\frac{1}{|z_+-q_\pm|} \leq \,\, \frac{1}{|x-u|} \lesssim  \,\, \begin{cases}
                                      \frac{1}{x}, &\quad u\in I_l(x), \\
                                      \frac{1}{u}, &\quad u\in I_r(x)
                                     \end{cases}
\ee
and the corresponding integrals are bounded in $\Lcal_{p, \beta}$ by Hardy's inequality. More precisely, 
$$
\bigg|\int_{\Ga^\pm_l(x)} \om(q)\, \bigg(\frac{z'_+}{z_+-q}\bigg)\,ds_q\bigg| \, \lesssim \, \frac{1}{ x}\int_0^{x} |\om(u)| + |\om(-u)| \, du
$$
\iffalse
$$
\aligned
\int_0^\de x^{\beta p} \bigg|\int_{\Ga_l(x)} \om(q)&\bigg(\frac{z'_+}{z_+-q}\bigg)\,ds_q\bigg|^p dx \, \lesssim \\
&\int_0^\de x^{(\beta - 1) p} \bigg(\int_{0}^{x} |\om(u)|+|\om(-u)| \, du \bigg)^p\, dx \, \lesssim \, \|\,\om\|_{p,\beta},
\endaligned
$$
where we have used Lemma \ref{lem:hardy}.\ref{itm:hardy1}  with $\ga := \beta - 1$ and $\lambda := \beta$ for the last inequality. 
\fi
belongs to $\Lcal_{p,\beta}$ by the first part of Lemma \ref{lem:hardy} with $\ga := \beta - 1$ and $\lambda := \beta$, while
$$
\bigg|\int_{\Ga^\pm_r(x)} \om(q)\, \bigg(\frac{z'_+}{z_+-q}\bigg)\,ds_q\bigg| \, \lesssim \, \int_{x}^{\de}\frac{|\om(u)| + |\om(-u)|}{u} \, du,
$$
\iffalse
$$
\aligned
\int_0^\de x^{\beta p} \bigg|\int_{\Ga_r(x)} \om(q)&\,\bigg(\frac{z'_+}{z_+-q}\bigg)\,ds_q\bigg|^p dx \, \lesssim \\
&\int_0^\de x^{\beta p} \bigg(\int_x^{2\delta} \bigg|\frac{\om(u)}{u}\bigg| + \bigg|\frac{\om(-u)}{u}\bigg| \, du \bigg)^p\, dx \, \lesssim \, \|\,\om\|_{p,\beta},
\endaligned
$$
where we have used Lemma \ref{lem:hardy}.\ref{itm:hardy2} with $\ga := \beta + 1$ and $\lambda := \beta$. 
\fi
belongs to $\Lcal_{p,\beta}$ by the second part of Lemma \ref{lem:hardy} with $\ga := \beta + 1$ and $\lambda := \beta$. 

Let now $q\in \Ga^+_c(x)$. A short calculation and the mean value theorem give 
\be\label{eq:eqga+}
\frac{z_+'}{z_+ - q_+} = \frac{1}{x-u} + \underbrace{\frac{i}{x-u}\bigg(\ka'(x) - \frac{\ka(x) - \ka(u)}{x-u}\bigg)}_{\text{$\frac{i\ka_+''(\xi)}{2}$ for some $\xi\in [u,x]$}}\underbrace{\frac{1}{1 + i \frac{\ka(x) - \ka(u)}{x-u}}}_{\text{$\frac{1}{1 + i\ka'_+(\eta)}$; $\eta\in[u,x]$}},
\ee
where we have set $[u,x]:= [\min\{u,x\},\max\{u,x\}]$. In particular, since $u$ and $x$ are comparable and $\|\ka\|_{\Ccal^2_{1-\mu}}<\infty$, we have
\be\label{eq:ga+}
\frac{z'_+}{z_+-q_+} = \frac{1}{x-u} + O( x^{\mu - 1})
\ee
and the integral over the error term can be treated by the Hardy's inequality as was done for $\Ga^\pm_l(x)$. On the other hand, the Hilbert transform is bounded on $\Lcal_{p,\beta}$ as long as $|x|^\beta$ defines a M\"{u}ckenhaupt weight. 

Finally, let $q\in \Ga^-_c(x)$. Then, setting $\t z = x + i \ka(u)$ we have
\be\label{eq:eqga-}
\frac{z_+'}{z_+ - q_-} = \frac{1}{\t z - q_-} + \frac{i\ka'(x)}{\t z - q_-} + \underbrace{iz_+' \frac{\ka(u)-\ka(x)}{(z_+ - q_-)(\t z - q_-)}}_{|\cdot|\, \lesssim \, \frac{|\ka'(\xi)|}{|\t z - q_-|}\text{ for some $\xi\in [u,x]$} }.
\ee
In particular, since $\sup_x|x^{-\mu}\ka'(x)|<\infty$ we have
\be\label{eq:ga-}
  \frac{z_+'}{z_+ - q_-} = \frac{1}{\t z -q_-} + O\bigg( \frac{x^{\mu}}{|\t z - q_-|}\bigg),
\ee
with the error term of order $O(x^{-1})$ which can therefore be treated by the Hardy's inequality as before. For the remaining term, we employ the variable change 
$$
u=h(\tau), \, h'(\tau) = -\rho(h(\tau))
$$ 
on the interval $[0, \de]$ (cf. Appendix \ref{S.appendix}). We have
$$
\aligned
\int_0^{\de/2} x^{\beta p} \bigg|\int_{(1-\epsilon)x}^{(1+\epsilon)x} \om(u)\,&\frac{du}{(x-u) + i\rho(u)}\bigg|^p  \,dx \, \lesssim\\
&\int_0^\infty \underbrace{h(\xi)^{\beta p} |h'(\xi)|}_{\sim \xi^{-\al p}}\bigg|\int^{\ga^+(\xi)}_{\ga^-(\xi)}\om(h(\tau)) \underbrace{ \frac{h'(\tau)}{(h(\xi)-h(\tau)) - ih'(\tau)} }_{=:k(\xi, \tau)}\,d\tau \bigg|^p\,d\xi 
\endaligned
$$
where we have set
$$
\al:=\mu^{-1}(\beta + p^{-1}) + p^{-1}
$$ 
and $\ga^\pm(\xi) = h^{-1}((1\mp\epsilon)x)$ are comparable, i.e. the corresponding interval is contained in $I_{\t \epsilon}(\xi):= \{\tau\in \Rbb_+ \, : \, |\tau - \xi| < \t \epsilon \,\xi\}$ for some $\t \epsilon > 0$. Setting 
$$
F(\xi, \tau) := \frac{h(\xi) - h(\tau)}{\xi - \tau}
$$
we isolate the singular part of the kernel $k(\xi, \tau)$ via 
$$
k(\xi, \tau) - \frac{1}{(\xi-\tau) - i} = \underbrace{i\bigg( \frac{h'(\tau)}{F(\xi,\tau)} - 1\bigg)\frac{\frac{h'(\tau)}{F(\xi,\tau)} }{(\xi - \tau)-i\frac{h'(\tau)}{F(\xi,\tau)} }\frac{1}{(\xi - \tau) - i}}_{|\,\cdot\,|\, \lesssim \, \frac{1}{\tau}}
$$
where we have used 
$$
\frac{1}{(\xi -\tau)}\bigg(\frac{h'(\tau)}{F(\xi,\tau)} - 1\bigg) =-\frac{F_\tau(\xi ,\tau)}{F(\xi, \tau)}
$$
which for any $|\tau - \xi |\leq \t\epsilon\, \xi$ satisfies 
$$
\frac{F_\tau(\xi ,\tau)}{F(\xi, \tau)} \sim \frac{h''(\tau)}{h'(\tau)} \sim \frac{1}{\tau}.
$$
In particular,
$$
\aligned
\int_0^\infty \xi^{-\al p}&\bigg|\int_{I_{\t\epsilon}(\xi)}\om(\tau)\,k(\xi, \tau) d\tau \bigg|^p\,d\xi \, \lesssim \, \\
&\int_0^\infty \xi^{-\al p}\bigg|\int_{I_{\t\epsilon}(\xi)}\om(\tau)\,\frac{1}{(\xi -\tau) - i} \, d\tau \bigg|^p d\xi + \int_0^\infty \xi^{-\al p}\bigg(\int_{I_{\t\epsilon}(\xi)}\frac{|\om(\tau)|}{\tau}d\tau \bigg)^p d\xi,
\endaligned
$$
where we have set $\om(\tau):= \om(h(\tau))$. The second term is bounded by the Hardy's inequality, i.e. by part \ref{itm:hardy1} of Lemma \ref{lem:hardy}, since 
$$
(1-\al) + p^{-1} <1-\mu^{-1}(\beta + p^{-1}) <1
$$
and 
$$
\om\circ h\in \Lcal_{p,-\al}(\Rbb_+)\quad  \Leftrightarrow \quad \om\in\Lcal_{p,\beta}\big([0,\de]\big).
$$ 
To show the first term is controlled by $\|\om\circ h\|_{-\al, p}$, we want to apply Lemma \ref{lem:aux1}, but $-\al$ does not satisfy the required assumptions, since $-\al + p^{-1} < 0$. However, there exists a positive real number $M > 0$ such that  
$$
0<(M -\al) + p^{-1} < 1
$$
i.e. $(M - \al)$ is a Muckenhaupt weight. Note that we are allowed to put $(\xi/\tau)^M$ in the inner integral since $\xi$ and $\tau$ are comparable. Moreover, we have $\frac{\om(\tau)}{\tau^M}\in \Lcal_{p, M-\al}(\Rbb)$. In particular, the claim follows from Lemma \ref{lem:aux1} with $\ga = M - \al$.  

Let now $\om\in\Wcal^{1,p}_{\beta}$ and assume for simplicity the interface is parametrized w.r.t. the arc-length parametrization. We need to show 
$$
\big(BR(\om,z)^*z'\big)'\in \Lcal_{p,\beta}(\Ga).
$$
In fact, it is enough to show the derivative of $BR(z,\om)^*$ belongs to $\Lcal_{p,\beta}$. Then, by the first part of the proof we have $BR(z,\om)^*\in \Wcal^{1,p}_{\beta}$ and in particular $BR(z,\om)^*\in L^\infty$, while the condition $1-\mu < \beta + p^{-1}$ implies
$$
\Ccal_{1-\mu}(\Ga)\hookrightarrow \Lcal_{p,\beta}(\Ga)
$$  
that is, we have $z''\in \Lcal_{p,\beta}$. Since the interface is twice continuously differentiable except possibly at the splash point and the cusps are 'smoothly' connected, for any $z\neq z_*$ integration by parts gives as usual  
$$
\frac{d}{ds}BR(z,\om)^* =  \frac{1}{2\pi i}\int_\Ga \bigg(\frac{\om(q)}{q'}\bigg)'\bigg(\frac{z'}{z-q}\bigg)\,ds_q. 
$$
However, this is bounded in $\Lcal_{p,\beta}$ by the first part of the proof, since we have
$$
\bigg(\frac{\om(q)}{q'}\bigg)' = \frac{\om'(q) - \om(q)q''/q'}{q'}\in \Lcal_{p,\beta}
$$
(in the arc-length parametrization we have $|q''/q'|\sim |\theta'|\in \Ccal_{1-\mu}(\Ga)$, where $\theta$ denotes the tangent angle).

\end{proof}

Taking the real part of the above complex product, we recover the tangential part of the Birkhoff-Rott integral. In what follows, we denote this operator by $S$, i.e. 
$$
S\om :=  2\pi BR(z, \om)\cdot z' = 2\pi \Re\big(BR(\om,z)^*z'\big) = p.v.\int_\Ga \om(q) \, \pa_{n_z}\log|z - q|^{-1} ds_q.
$$
We will also need the double-layer potential operator   
$$
T\om(z) := p.v.\int_\Ga \om(q) \, \pa_{n_q}\log|z - q|^{-1} ds_q = - 2\pi \Re\bigg(\frac{1}{2\pi i} \int_\Ga \om(q)\frac{dq}{z-q}\bigg).
$$
\begin{proposition}\label{prop:doublelayer}
Let $0<\beta + p^{-1}<1$. Then, the operator 
$$
T: \Wcal^{1,p}_\beta(\Ga) \rightarrow \Wcal^{1,p}_\beta(\Ga) \\
$$
is bounded.
\end{proposition}
\begin{proof}
Given $\om\in \Wcal^{1,p}_\beta(\Ga)$ this operator satisfies
$$
(T\om)' = -S\om'
$$
hence proof of Proposition \ref{prop:BR} implies it is bounded on $\Wcal^{1,p}_\beta(\Ga)$. 
\end{proof}

\iffalse
\begin{remark}\label{rem:regST} In fact, more can be said for $\pi I + T$ respectively $\pi I + S$, then for $T$ or $S$ alone. We actually have
$$
\pi I + T:\Wcal^{1,p}_{\beta}(\Ga) \longrightarrow \Ncal^{1,-}_{p,\beta}(\Ga), \qquad 0<\beta + p^{-1}<\min\{\mu, 1\}
$$
and
$$
\pi I + S:\Wcal^{1,p}_{\beta}(\Ga) \longrightarrow \Ncal^{1,-}_{p,\beta}(\Ga)\oplus\Ical^1(\Ga), \qquad  1-\mu < \beta + p^{-1}<\mu.,
$$
where $\Ical^1(\Ga)$ is a space of dimension 1.  
More precisely, for e.g. the operator $\pi I + S$, we have the following decomposition 
$$
(\pi I + S)\om(z_\pm) = \pi (\om(z_+) + \om(z_-)) + R\om(z_\pm) \pm \sum_{k=0}^1 t(z_\pm)\cdot c_k(\om)\, x^k,
$$
where $R : \Wcal^{1,p}_{\beta}(\Ga) \longrightarrow \Lcal_{p, \beta -\mu-1 }(\Ga)$ and $c_k:\Lcal_{p,\beta-1}(\Ga)\longrightarrow \Rbb$ are continuous linear functionals 
$$
c_k(\om) := \int_\Ga \om(q)\frac{(q^k)^\perp}{|q^k|^2}ds_q
$$  
and similarly for $\pi I + T$ with the only difference that in this case  ``far''-contributions cancel out when taking the difference. The proof is essentially the same as the one given in \cite{MazSo}. We omit the details, as we are only interested in the space of even $\om$, in which case everything simplifies to $\Wcal^{1,p}_{\beta\text{(even)} }(\Ga)$.
\end{remark}
\fi

We finish this section with a compactness result for the difference $S-T$. 

\begin{proposition}\label{prop:compact}
Let $1-\mu< \beta + p^{-1} < \mu$. Then,  
\begin{align}
S-T:& \Lcal_{p,\beta}(\Ga)\longrightarrow \Lcal_{p,\beta}(\Ga) \label{eq:ST} \\  
S-T:& \Wcal^{1,p}_\beta(\Ga)\longrightarrow \Wcal^{1,p}_\beta(\Ga) \label{eq:STder}
\end{align}
are compact.
\end{proposition}
\begin{proof}
If the interface is a sufficiently regular chord-arc curve, then $T$ and $S$ are both compact operators of Hilbert-Schmidt type. In particular, it is enough to consider $z$ in a small neighborhood of the splash point, i.e. we may assume $z = z_+\in \Ga^+$ with $0 < x < \de/2$. Taking into account the change of orientation on the upper branch of $\Ga\cap B(z_*)$ when parametrized as a graph, we obtain
$$
(S-T)\om(z_+) = p.v.\int_{\Ga}\, \om(q)\Big\{-\pa_{n_{z_+}}\log|z_+-q|^{-1} - \pa_{n_q}\log|z_+-q|^{-1}\Big\}\,ds_q.
$$
Moreover, we may further restrict attention to $q= q_-\in \Ga^-$ (when the integral runs over $\Ga \setminus (\Ga^- \cup \Ga^+)$ the kernel of both operators is bounded, while it is at most weakly singular on $\Ga^+$). In particular, it is enough to consider the operator
\iffalse
$$
\aligned
\om \longrightarrow \int_{\Ga^+} \om(q)\,\Im\bigg\{-\bigg(\frac{z'_+/|z'_+|}{z_+ - q_+}\bigg)& - \bigg(\frac{q'_+/|q'_+|}{z_+ - q_+}\bigg) \bigg\}\,ds_q \\
+ \int_{\Ga^-} &\om(q)\,\Im\bigg\{-\bigg(\frac{z'_+/|z'_+|}{z_+ - q_-}\bigg) + \bigg(\frac{q'_-/|q'_-|}{z_+ - q_-}\bigg) \bigg\}\,ds_q
\endaligned
$$
\fi
$$
\om(x) \longrightarrow \Im \int_{\Ga^-} \om(u)\,\underbrace{\bigg (- \frac{|q'_-|}{|z'_+|}\frac{z'_+}{z_+ - q_-} + \frac{q'_-}{z_+ - q_-} \, \bigg)}_{:=k(x,u)} du.
$$
The kernel reads
$$
k(x,u) = -i\big( \ka'(u) + \ka'(x)\big)\bigg(\frac{1}{z_+ - q_-}\bigg) + \bigg(1- \frac{|q_-'|}{|z_+'|}\bigg)\frac{z_+'}{z_+ - q_-}
$$
with
$$
\bigg(1- \frac{|q_-'|}{|z_+'|}\bigg) = O\big(\ka'(x)^2 - \ka'(u)^2\big).
$$
In particular, using estimate \eqref{est:lrbasic} it is not difficult to see the integral over $\Ga^-_l(x)$ is dominated by the compact operator  
\be\label{op:HL}
\om \longrightarrow x^{\mu -1}\int_0^x (|\om(u)| + |\om(-u)|)du \,:\, \Lcal_{p,\beta}\longrightarrow \Lcal_{p,\beta},
\ee
while the integral over $\Ga^-_r(x)$ is dominated by the compact operator 
\be\label{op:HR}
\om \longrightarrow \int_x^{\delta} \big(|\om(u)|+ |\om(-u)|\big)\, u^{\mu - 1}du \,:\,\Lcal_{p,\beta}\longrightarrow \Lcal_{p,\beta}
\ee
(cf. \ref{S.appendix}. Lemma \ref{lem:hardy} in both cases). It remains to estimate the integral over $\Ga^-_c(x)$. We have 
$$
k(x,u) = -i\rho'(u)\,\frac{1}{z_+ - q_-} + O(x^{\mu - 1})
$$
with the integral over the error term essentially dominated by \eqref{op:HL} and therefore compact. On the other hand, a calculation similar to \eqref{eq:eqga-} implies  
$$
\rho'(u)\frac{1}{z_+ - q_-} = \rho'(u)\frac{1}{(x-u) + i\rho(u)} + O (x^{\mu - 1})
$$
and we are finished, since compactness of the remaining term follows from Lemma \ref{lem:compact}.  

It remains to consider the derivative of $S- T$. For simplicity, we temporarily assume the interface is parametrized w.r.t. the arc-length parametrization and we show 
$$
\frac{d}{ds}(S-T):\Wcal^{1,p}_{\beta} \longrightarrow \Lcal_{p,\beta}
$$
is compact.  For any $z\neq z_*$, we have
$$
\aligned
(S\om)' &=  \Im\Bigg(z''\int_\Ga \frac{\om(q)}{q'} \frac{q'}{z-q}\,ds_q\Bigg) +  \Im\Bigg(z'\int_\Ga \bigg(\frac{\om(q)}{q'}\bigg)' \frac{z'}{z-q}\,ds_q\Bigg)\\
(T\om)' &= - S\om ' = - \Im\Bigg(\int_\Ga \om'(q) \frac{z'}{z-q}\,ds_q\Bigg)
\endaligned
$$
(where all the integrals are understood in the sense of principal value). By the first part of the proof $(S - T)\om'$ is compact, hence it remains to consider 
$$
\aligned
(S\om)' - T\om' &= \Im\Bigg(\int_\Ga \om(q) \, \frac{z'' - q''}{z-q}\,ds_q + \int_\Ga \Big(\om'(q)/q' - \om(q)q''/(q')^2\Big)\frac{(z')^2 - (q')^2}{z-q}\,ds_q\Bigg) \\
&=: I_1 + I_2.
\endaligned
$$
The kernel of $I_2$ is at most weakly singular, except when e.g. $z = z_+\in \Ga^+$ and $q = q_-\in \Ga^-$. However, there the kernel (in the graph parametrization) reads
$$
\bigg(\frac{z'_+}{|z'_+|} - \frac{q'_-}{|q'_-|}\bigg) \bigg(-\frac{|q'_-|}{|z'_+|}\frac{z'_+}{z_+ - q_-} + \frac{q'_-}{z_+ - q_-}\bigg) 
$$
i.e. we recover $k(x,u)$ from above multiplied by a bounded function. Since clearly  $\om'(q)/q' - \om(q)q''/(q')^2\in \Lcal_{p,\beta}$, this can be estimated as in the first part of the proof. Finally, $-I_1$ is the imaginary part of 
$$
\om(0) \, \underbrace{\int_\Ga  q''\, \frac{ds_q}{z-q} }_{\in  \Lcal_{p,\beta}(\Ga)} + \int_\Ga (\om(q) - \om(0)) q'' \frac{ds_q}{z-q} - z'' \underbrace{\int\om(q) \, \frac{ds_q}{z-q}}_{\in \Wcal^{1,p}_\beta}
$$ 
The first term is an operator of finite rank and is therefore compact. The second term belongs to $\Wcal^{1,p}_{\beta + 1 -\mu}(\Ga)$, since $q''(\om(q) - \om(0))\in \Wcal^{1,p}_{\beta + 1 -\mu}(\Ga)$ and $\beta + 1 - \mu$ defines a Muckenhaupt weight, cf. Proposition \ref{prop:BR}. In particular, it can be written as a sum of a continuous linear functional on $\om$ and an integral of a function in $\Lcal_{p, \, \beta + 1 - \mu}$ which is compact in $\Lcal_{p, \beta}$ by Lemma \ref{lem:hardy}. Similarly, we conclude the last term is compact.  
%
%
\iffalse
On the other hand, the only singular term in $I_1$ is given by the imaginary part of $z'' - q''$ which, when considered w.r.t. to the graph parametrization, reads 
$$
\frac{\ka''(x)}{|z'(x)|^2} + \frac{\ka''(u)}{|q'(u)|^2}
$$
(the real part does not change the sign when we change the orientation on $\Ga^+$, so that part of the kernel is weakly singular). Since $\ka''$ is H\"{o}lder continuous, we are left with
$$
\ka''(x) \Re \bigg( \int_{\Ga^-} \om(u) \, \frac{1}{z_+-q_-}du\bigg),
$$
however this integral and its derivative belong to $\Lcal_{p,\beta}$ on $[0,\de/2]$, hence it is compact in $\Lcal_{p,\beta}$ as well.  
\fi
\end{proof}

\end{subsection}

\end{section}

\section{Results on invertibility}\label{S.invertibility}

 In this section we study invertibility of the operator $\pi I + S$
 which corresponds to solving the Neumann problem for the Laplace
 equation in the ``splash'' domain. We follow Maz'ya and Soloviev (cf. \cite{MazSo}) and use conformal maps to transform singular domains to the horizontal strip where solutions of the Laplace equation are well known. %= \pi\big(\om + 2BR(z, \om)\cdot z_\al/|z_\al|\big)$corresponding to solving the Neumann problem in the periodic "splash"-domain $\Om_p$. 
 
 Unlike the previous section, here we work with periodic versions of the relevant singular integrals, i.e. $\Ga$, $\Om$ and $\Om^c$ are as defined in section \ref{S.setting}. However, when we construct the required harmonic functions in section \ref{S.aux} we will work with bounded domains, hence we introduce the following notation for their images under the exponential map
 $$
 P: \Scal_\pi \rightarrow \Cbb\setminus [0, \infty), \qquad P(z):= -e^{-iz} .
 $$  
 The interface $\Ga$ is then mapped to a closed, bounded curve $\t\Ga=\overline{P(\Ga)}$, while
 $$
 P(\Om) = \t \Om\setminus [0, \infty), \qquad P(\Om^c) = \t \Om^c\setminus [0, \infty), 
 $$
 where $\t\Om$ corresponds to the interior of $\overline{P(\Om)}$ and $0\in\t\Om$.  
 
 We first state a result on the boundedness of the periodic Cauchy integral and the periodic single-layer potential on $\Rbb^2$ away from the interface. We omit the proof, as it is a straightforward adaptation of a similar result given in \cite{MazSo}.

\begin{lemma}
\label{lem:Vc}
Let $0<\beta + p^{-1}<1$ and let $\om\in \Lcal_{p, \beta}(\Ga)$. Then, the periodic, complex single-layer potential  
$$
\Vcal\om(z) = -\frac{1}{2\pi i}\int_\Ga \om(q) \log \Big(\sin \Big( \frac{z-q}{2}\Big)\Big) ds_q 
$$
is bounded on any horizontal strip containing $\Ga$ and we have
\be\label{eq:vcalinf}
\Vcal\om(z) = \Im z + O(1)\qquad \Im z\gtrless \infty, \quad |z|\rightarrow \infty.
\ee
If $\om\in \Wcal^1_{p, \beta}(\Ga)$, then the periodic Cauchy integral   
$$
\Wcal \om(z) = \frac{1}{4\pi i}\int_\Ga \om(q)\cot\Big(\frac{z - q}{2}\Big)\,dq 
$$
is bounded on $\Rbb^2$. More precisely, we have
$$
\big|\Wcal\om(z)\big| \leq C, \qquad \forall z\in\Rbb^2
$$
with
\be\label{eq:wcalpinf}
\Wcal\om(z)  \pm \frac{1}{4\pi}\int_\Ga \om(q)dq \longrightarrow 0, \qquad \Im z\gtrless \infty, \quad |z|\rightarrow \infty.
\ee
\end{lemma}

%\begin{proof}
 
%\end{proof}

We now construct solutions to the exterior Dirichlet problem in the
``cusp''-domain $\Om^c$.

\begin{theorem}\label{thm3}Let $0<\beta + p^{-1}<\mu$. Then
$$
\pi I + T: \Wcal^{1,p}_{\beta \rm (even)}(\Ga) \longrightarrow \Wcal^{1,p}_{\beta \rm (even)}(\Ga) 
$$
is surjective.
\end{theorem}

\begin{proof}
 Let $\phi \in \Wcal^{1,p}_{\beta \rm (even)}(\Ga)$. By Proposition \ref{prop:in}, there exists a bounded harmonic function such that  
 $$
 \begin{cases}
\Delta u^e = 0& \quad \mathrm{in } \,\, \Om^c, \\
 u^e = \phi & \quad \mathrm{on } \,\, \t \Ga\setminus\{z_*\}
\end{cases}
 $$
and similarly by Proposition \ref{prop:out} there exists a harmonic function on $\Om$ such that
 $$
 \begin{cases}
\Delta u^i = 0& \quad \mathrm{in } \,\,\Om, \\
 \pa_n u^i = \pa_n u^e& \quad \mathrm{on } \,\, \t \Ga\setminus\{z_*\}.
\end{cases}
 $$
Let $W\om = 2\pi \Re \Wcal\om$ and let $\Phi$ be the harmonic function 
$$
\Phi(z) := \begin{cases}
            \Phi^i(z) := W\om(z) - u^i(z) + u^i(\infty),& \quad z\in \Om  \\
            \Phi^e(z) := W\om(z) - u^e(z) + u^e(\infty),& \quad z\in\Om^c,  \\
           \end{cases}
$$
with $u^{e, i}(\infty) := \lim_{\Im z\rightarrow \pm \infty} u^{e,i}(z)$. We define $\om$ in such a way that   
\begin{equation}
\label{eq:comp}
\Phi^i(z) = \Phi^e(z),\quad \pa_n \Phi^i(z) = \pa_n \Phi^e(z) \quad\quad \forall z\in\Ga\setminus\{z_*\}.
\end{equation}
The second condition is satisfied by construction (the normal derivative of $W\om$ is continuous over the boundary), while the first one follows from jump relations for $W\om$, provided that 
$$
\om := \frac{1}{2\pi}(u^e - u^e(\infty) - u^i + u^i(\infty)).
$$
By adding an overall constant to $u^i$, we may assume $\Re\int_\Ga\om(q)dq=0$ and therefore 
$$
\Phi(z) \rightarrow 0, \qquad |z|\rightarrow \infty.
$$ 
We extend $\Phi$ periodically to all of $\Rbb^2\setminus \Ga$. The compatibility conditions \eqref{eq:comp} then imply $\Phi$ can be extended to a continuous function on $\Rbb^2\setminus\{z_* + 2k\pi\}$ for all $k\in \Zbb$. Moreover, it is not difficult to see that $\Phi$ also satisfies the mean-value property. In fact, let $z_0\in\Ga\setminus\{z_*\}$ and let $B \equiv B_r(z_0)$ be so small that $B\subseteq \Rbb^2\setminus\{z_* + 2k\pi\} $. In the interior of $B^i = B\cap \Om$ the harmonic function $\Phi^{i}$ (respectively $\Phi^e$ in the interior of $B^e = B\cap \Om^c$) satisfies
$$
\Phi^{i,e}(z) = \frac{1}{2\pi}\int_{\pa B^{i,e}} \Big(-\Phi^{i,e}(q)\frac{\pa}{\pa n_q} \log\frac{1}{|z-q|} + \frac{\pa\Phi^{i,e}}{\pa n}(q) \log \frac{1}{|z-q|} \Big)\, ds_q,  
$$
so letting $z\rightarrow z_0^\pm$ and using \eqref{eq:comp}, we get
$$
\Phi(z_0) = \frac{1}{2\pi}\int_{\pa B} \Big(-\Phi(q)\frac{\pa}{\pa n_q} \log\frac{1}{|z_0-q|} + \frac{\pa\Phi}{\pa n}(q) \log \frac{1}{|z_0-q|} \Big)\, ds_q.
$$
Since 
$$
\log|z_0 - q| = \log r, \quad \frac{\pa}{\pa n_q} \log\frac{1}{|z_0-q|} = \frac{1}{r}, \quad \int_B\pa_n \Phi =0
$$
we have 
$$
\Phi(z_0) = \frac{1}{2\pi r} \int_{\pa B} \Phi(q)\,ds_q.
$$
In particular, $\Phi$ is harmonic on $\Cbb\setminus\{z_* + 2k\pi\}$, vanishes at infinity and has a removable singularity at each $z_* + 2k\pi$, since $W\om$, $u^i$ and $u^e$ are all $O(1)$ there. Therefore $\Phi\equiv 0$ and 
$$
(\pi I + T)\om = \phi.
$$
\end{proof}

\begin{theorem}
\label{thm2}
Let $0<\beta + p^{-1}<\mu$. Then the operator 
$$
\pi I + T: \Wcal^{1,p}_{\beta \rm (even)}(\Ga) \longrightarrow \Wcal^{1,p}_{\beta \rm (even)}(\Ga) 
$$
is injective.  
\end{theorem}

\begin{proof}
Let $\om\in \Wcal^{1, p}_{\beta \rm (even)}(\Ga)$ be such that $(\pi I + T)\om = 0$. By construction, we have 
$$
(\pi I + T)\om (z_0) = \lim_{z\rightarrow z_0 ^-} W\om(z), \qquad  z_0\in \Ga\backslash \{z_*\}, \, z\in \Om^c  
$$
where $W\om  :=  2\pi \Re (\Wcal\om)$.
\iffalse
In fact, since $\om(z_*) = 0$ at the remaining point $z_*$ we actually have 
$$
\lim_{S\setminus \Ga \ni z\rightarrow z_*} W\om(z) = T\om(z_*)
$$
\fi
In particular, $W\om$ solves  
$$
\begin{cases}
 \Delta W\om = 0& \quad \mathrm{in } \quad \Om^c, \\
 W\om = 0 & \quad \mathrm{on } \quad \Ga\setminus{\{z_*\}}.
\end{cases}
$$
Since $W\om$ does not converge to $(\pi I +T)\om$ when approaching the singular point (unless $\om(z_*)=0$), we map each connected component of $\Om^c$ to the horizontal strip via the mapping $F$ constructed in Lemma \ref{lem:r}. For the unbounded component, some care is needed. Since $\Wcal_p\om$ is $2\pi$-periodic, we can extend $\Wcal_p \om \circ P^{-1}$ as an analytic function on $\t\Om_2$, a type of domain treated in Lemma \ref{lem:r}. In both cases, the resulting mapping is a harmonic function on $\Pi$ which vanishes on $\pa\Pi$. Moreover, by Lemma \ref{lem:Vc} it is bounded, hence we may use the Phraghmen-Lindelof principle on the strip (e.g. \cite{Ran}) to conclude
\be\label{eq:Win}
W\om\equiv 0 \quad \text{in } \Om^c, \qquad \Re\int_\Ga \om(q) dq = 0.
\ee
In particular, $\t W\om  = 2\pi \Im (\Wcal_p\om)$ must be equal to a (not necessarily same) constant in each connected component of $\Om^c$. Since $\t W\om$ is continuous when crossing $\Ga$ and its limit is continuous on $\Ga$ (its derivative belongs to $\Lcal_{p,\beta}(\Ga)$) it solves the Dirichlet problem in $\Om$ with constant boundary data on all of $\Ga$ including $z_*$. Again by the Phragmen-Lindelof principle (this time we may consider the entire domain $\Om^p$) we conclude 
$$
\t W\om \equiv c \quad \text{in } \Om
$$ 
and by the Cauchy-Riemann equations the same is true for $W\om$. However, \eqref{eq:Win} and \eqref{eq:wcalpinf} imply that it actually vanishes, hence $\om \equiv 0 $ on $\Ga$. 
\end{proof}

Finally, we are ready to prove the main theorem of this section. We have

\begin{theorem}
 \label{thm1}Let $1-\mu<\beta + p^{-1}<\mu$. Then, the operator  
 $$
 \pi I + S\, :\, \Wcal^{1, p}_{\beta \rm (even)}(\Ga) \longrightarrow \Wcal^{1, p}_{\beta \rm (even)}(\Ga)
 $$
is invertible.
 \end{theorem}
\begin{proof}
We have
$$
\pi I + S = (\pi I + T) + (S-T),
$$
where $\pi I + T$ is invertible by Theorems \ref{thm3} and \ref{thm2}, while
$$
S-T:\Wcal^{1,p}_{\beta(\rm even)}(\Ga) \longrightarrow \Wcal^{1,p}_{\beta(\rm even)}(\Ga)
$$
is compact by Proposition \ref{prop:compact}. In particular, $\pi I + S$ is a Fredholm operator of index 0 and therefore invertible if injective. So let $\om\in W^{1,p}_{\beta \rm(even)}(\Ga)$ be such that $(\pi I + S)\om = 0$. Then 
  $$
\begin{cases}
 \Delta V\om = 0& \mathrm{in } \,\,\Om, \\
 \pa_n V\om = 0 &\mathrm{on } \,\, \Ga\backslash\{z_*\},
  \end{cases}
  $$
where $V\om(z) =  \Re \Vcal\om(z)$. 
%In particular, its conjugate function must be identically equal to a constant on $\Ga$ (the constant is a priori different on each connected component, however its derivative is bounded by assumption). 
Since $\Vcal\om(z) = \Im z + O(1)$ as $\Im z \rightarrow -\infty$ the mapping 
$$
P\circ \Vcal\om \circ P^{-1}: \t\Om\setminus[0, \infty) \rightarrow \Cbb\setminus[0, \infty)
$$
can be extended as an analytic map to the entire domain $\t \Om$. Moreover, this mapping is bounded by Lemma \ref{lem:Vc}, hence if $F:\Pi \rightarrow \t\Om$ denotes the mapping constructed in Lemma \ref{lem:r} we must have
$$
P\circ \Vcal\om \circ P^{-1}\circ F \equiv c\in \Cbb
$$  
by the Phraghmen-Lindelof principle. Now $V\om$ is continuous over the boundary and therefore satisfies 
  $$
\begin{cases}
 \Delta V\om = 0& \mathrm{in } \,\,\Om_j, \\
 V\om = \Re c &\mathrm{on } \,\, \Ga_j\setminus\{z_*\},
  \end{cases}
  $$
hence we may repeat the procedure from above with mappings $F_j:\Pi \rightarrow \t \Om_j$ to conclude that $V\om$ must be identically equal to a constant on $\Om_j$ and therefore 
$$
2\pi \om = \pa_{n^+}V\om - \pa_{n^-}V\om \equiv 0. 
$$
\end{proof}

\section{Proof of Theorem~\ref{T.main}}
\label{S.IFT}

In this section we present the proof of Theorem~\ref{T.main}, thereby
extending our results for stationary solutions with two fluids
obtained in~\cite{CEG} to curves which exhibit splash
singularities. More precisely, we use the implicit function theorem
around the critical Crapper solution $\theta_A$ to construct two-fluid
splash solutions (with small, but positive upper fluid density $\rho_1$). 

Let $G = (G_1, G_2)$ be a mapping defined by
\begin{flalign*}
G_1(\theta, \om; \,\ep, g, \kappa) :=& q \Big(1 +
                                          \frac{\ep}{2}\Big)\frac{d\theta}{d\al}
                                          + \sinh H\theta - \frac{\ep}{4}e^{H\theta}\om (\om - 2) 
                                          +e^{-H\theta}\big(g \,\Imag I(\theta) - \ka\big),\\
G_2(\theta, \om; \,\ep, g, \kappa) :=&  \om + 2BR(z,\om)\cdot \pa_\al z - 2 + \cos\theta,
\end{flalign*}
where $z$ is considered to be a function of $\theta$ via 
$$
z(\alpha) = I(\theta)(\al) := -\pi + \int_{-\pi}^\al e^{-H\theta(\al ') + i \theta (\al')} d\al'.
$$
We first give the set of admissible $\theta\in H^2_{odd}(\Tbb)$. Let $\al_*\in (0,\pi)$ be fixed. We define
$$
M(\al_*) := \{ \theta\in H^2_{odd}(\Tbb) \, : \quad  \theta(\al_*) =  \pi/2, \quad   \int_0^{\al_*} e^{-H\theta(\al ') + i \theta (\al')} d\al' = 0\},
$$
i.e. we assume the corresponding curve $z(\al)$ has a pair of vertical tangents at $z(\pm\al_*) = 0$. Moreover, we only consider $\theta \in M(\al_*)$ such that
\be\label{eq:open-set-conditions}
\Fcal_\de(z) < 0, \quad \quad \sup_{\al\in B_\de(\al_*)}\theta'(\al) < 0, \quad \quad \inf_{\al\in B_\de(\al_*)} z_{2\al}(\al) > 0
\ee  
for some small $\de>0$ (where $\Fcal_\de(z)$ has been defined in cf. \eqref{arcchord}). The first assumption implies the corresponding curve is arc-chord outside of $B_\de(\al_*)\times B_\de(-\al_*)$. The second and third condition ensure the curve has exactly one vertical tangent in $B_\de(\al_*)$ (resp. $B_\de(-\al_*)$ by symmetry), i.e. that it  satisfies \eqref{reg_assum} with $\mu = 1$ in the neighborhood of the origin. 

The set $M(\al_*)$ is not a linear space. The simplest way to analyze
functions $\theta\in M(\al_*)$ is through the Hilbert space
$$
E(\al_*) := \{ u\in H^2_{even}(\Tbb) \, : \quad u(\al_*) =  0, \quad
\int_0^{\al_*} [u(\al') + i Hu(\al ') ]\, d\al' = 0\}\,,
$$ 
which enables us to parametrize $M(\al_*)$ via the mapping
$$
\Psi: M(\al_*) \rightarrow E(\al_*), \quad \quad \Psi(\theta) := e^{-H\theta(\al)} \cos\theta(\al)\,.
$$
The function $\theta\in M(\al_*)$ and its Hilbert transform can be
recovered from $u\in E(\al_*)$ by inverting this mapping, which yields
the explicit formulas
$$
\theta = \arctan\Big(\frac{Hu}{u}\Big), \quad \quad  e^{-2H\theta} = u^2 + (Hu)^2.
$$
In order to apply the implicit function theorem, we shall consider $(G_1, G_2)$ as functions of $(u, \om; \epsilon, g, \kappa)$. 

From now on, we fix $\al_*\in(0,\pi)$ to be such that $z_A(\al_*) = z_A(-\al_*)$. Moreover, by shifting the coordinate system vertically, we may assume $z_A(\al_*) = 0$, in which case $\theta_A\in M(\al_*)$ and $u_A \equiv {z_A}_{1\al}\in E(\al_*)$. There is one more condition we need to impose on $u$. Given $\theta\in L^2(\Tbb)$ such that $\int^\pi_{-\pi}\theta(\al')d\al' = 0$, it can be shown that 
$$
\int^\pi_{-\pi} e^{-H\theta(\al ') + i \theta (\al')} d\al' = 2\pi
$$ 
(see \cite{OS} for a proof). Note that this fixes the behavior at infinity of the corresponding conformal extension to the lower-half plane. We therefore consider  
$$
\t E(\al_*) :=  \{ u \in E(\al_*) \, : \, \, \int_{-\pi}^{\pi} u(\al') d\al' = 0\},
$$  
where $\theta$ is a function of $u\in \t E(\al_*)$ via
\be\label{eq:u-to-theta}
\Psi(\theta) = u + u_A,  \quad \quad \Psi^{-1}(u) = \arctan\Big(\frac{Hu + Hu_A}{u + u_A}\Big).
\ee
The parametrization of the interface $z$ as a function of $u$ now reads
$$
z(\alpha) = I(u)(\al) := -\pi + \int_{-\pi}^\al [(u + u_A)(\al ') + i H(u + u_A)(\al')] d\al'.
$$
Finally, for the application of the implicit function theorem we further restrict to the open set $O(\al_*)\subseteq \t E(\al_*)$ of all $u$ such that $u + u_A$ satisfy \eqref{eq:open-set-conditions}. Then, the corresponding interface satisfies assumptions \eqref{reg_assum} with 
$$
\mu = 1
$$
and $G_2$ is well-defined within the framework of the previous section, provided that 
$$
\om \in \Wcal^{1,p}_{\beta(even)}(\al_*) \quad :\Leftrightarrow \quad \int_{0}^\pi(|\om(\al)|^p + |\om_\al(\al)|^p)|\al -\al_*|^{p\beta} d\al < \infty.
$$ 
Additionally, we require $\Wcal_{\beta}^{1,p} \subseteq H^1$, for then $G_1$ has values in $H^1$. This is satisfied if 
$$
0<\beta + p^{-1}<\frac{1}{2}.
$$ 
We claim     
\begin{equation}\label{defG}
G= (G_1, G_2): O(\al_*) \times \Wcal^{1,p}_{\beta(even)}(\al_*) \times \Rbb^3 \longrightarrow H^1_{even}(\Tbb) \times \Wcal_{\beta( even)}^{1,p}(\al_*)   
\end{equation}
satisfies the assumptions of the implicit function theorem. We emphasize that $G$ is to be understood as a function of $(u, \om;\, \ep, g, \kappa)$ as explained above.

Indeed, we denote points in the domain of definition of $G$ by
$$
\xi = (u, \om;\, \ep, g, \kappa),
$$ 
where for the critical Crapper solution (where $A = A_0$) we set  
$$
\xi_A = (0, \om_A;\, 0, 0, 0).
$$ 
By construction, we have 
$$
G(\xi_A) = 0,
$$
where $\om_A$ exists by Theorem \ref{thm1}. It is not difficult to see, using technique of section \ref{S.estimates}, that the Fr\'{e}chet derivative $D_{u, \om}G$ of $G$ exists and is continuous on the domain of definition of $G$. At $\xi_A$ it has triangular form  
\[ 
  D_{u, \om}G(\xi_A) = \left(
			  \begin{array}{cc}
                           \t \Ga & 0 \\
			   D_u G_2	& D_\om G_2   
                          \end{array}
		    \right)\,,
\]
where we have set $\t \Ga = D_u G_1$. The operator $(D_\om G_2)\si$ is exactly 
$$
\pi (D_\om G_2)\si = (\pi I + S(z_A))\si 
$$
since $G_2$ is linear in $\om$. By Theorem \ref{thm1}, we know that 
$$
D_\om G_2:\Wcal_{\beta(\rm even)}^{1,p}\rightarrow \Wcal_{\beta(\rm even)}^{1,p}
$$ 
is invertible. On the other hand, we have $\t\Ga = \Ga \circ D_u \Psi^{-1}(0)$ where $\Ga$ is the Fr\'{e}chet derivative of the pure capillary wave operator
$$
\Ga v = q \frac{dv}{d\al} + \cosh(H\theta_A)\, Hv.
$$
The structure of $\Ga: H^1_{odd} \rightarrow L^2_{even}$ is known; it is injective, but not surjective (see \cite{AAW} and \cite{OS} for a detailed discussion). The cokernel of $\Ga$ is spanned by $\cos\theta_A$. In particular, 
\be\label{eq:cokernel-gamma}
\langle\Ga v, \cos\theta_A\rangle = 0.
\ee
We prove an analogous result for $\t \Ga$:

\begin{lemma} The operator $\t \Ga: \t E(\al_*) \rightarrow H^1_{even}$ is injective, but not surjective. Its cokernel is spanned by $\cos\theta_A$.
\end{lemma}

\begin{proof}
First note that $\t \Ga = \Ga \circ D_u \Psi^{-1}(0)$ can be written as a compact perturbation of a Fredholm operator of index $-1$. Indeed, we apply $\Ga$ to   
$$
D_u \Psi^{-1}(0) v = e^{2H\theta_A}(u_A Hv - v H u_A) = e^{H\theta_A}(\cos\theta_A Hv - v \sin \theta_A)
$$
and collect all terms containing derivatives of $v$ in  
$$
\t \Ga_0 v := q e^{H\theta_A}(\cos\theta_A Hv' - v' \sin \theta_A).
$$
Then $\t \Ga - \t \Ga_0$ is compact, since the embedding $H^2 \rightarrow H^1$ is compact by the Rellich-Kondrachov theorem. 

It remains to show that $\t \Ga_0$ is Fredholm of index $-1$. Let $g\in H^1_{even}(\Tbb)$, then 
\be\label{eq:equation-gamma}
\t \Ga_0 v = g \quad \Leftrightarrow \quad \Im((v' + iHv')e^{-i\theta_A + H\theta_A}) = \frac{1}{q}e^{2H\theta_A}g. 
\ee
We claim that $\Im((v' + iHv')e^{-i\theta_A + H\theta_A})$ is the boundary value of a harmonic map defined on the domain $\Om_A$ corresponding to the Crapper solution $z_A$ (cf. \eqref{eq:crapper}) which vanishes as $y\rightarrow -\infty$. Indeed, $v + iHv$ is the boundary value of a function $W = W_1 + i W_2$ holomorphic on the lower half-plane which vanishes as $y\rightarrow -\infty$ and by construction $z_A$ has an inverse $w_A : \Om_A \rightarrow \Scal_\pi \cap \Cbb_-$. In particular, $W\circ w_A$ defines a holomorphic map on $\Om_A$ with complex derivative $(W\circ w_A)' = W'\circ w_A w_A'$. On the boundary $\pa \Om_A$, we have $w_A \circ z_A(\al) = \al$ and $w_A'\circ z_A = e^{-i\theta_A}$. Since $W'(\al, 0) = v' + i Hv'$, we conclude the imaginary part of $(W\circ w_A)'$ is the desired harmonic function. In particular, $\t \Ga_0$ is injective by uniqueness of the Dirichlet problem (in the class of solutions which vanish at infinity). On the other hand, the solution of the equation \eqref{eq:equation-gamma} exists by Proposition \ref{prop:in} (with $\beta = 0$ and $p = 2$) under the assumption that $\int_{-\pi}^\pi e^{2H\theta_A}g d\al= 0$. This last condition is necessary since $\Im(W\circ w_A)'(z_A(\al)) = d(W_2 \circ w_A)/d\al$ is even on $\pa\Om_A$. 

Finally, relation \eqref{eq:cokernel-gamma} implies the cokernel of $\t \Ga$ is also spanned by $\cos\theta_A$, since $\t \Ga$ is just the composition of the linear operator $\Ga$ with $D_u\Psi^{-1}(0)$.
\end{proof}

In particular, the $D_{\theta, \om}G(\xi_A)$ is injective, i.e. $\ker D_{\theta, \om}G(\xi_A) = 0$, but it is not surjective. In order to use the implicit function theorem in this situation, we use an adaptation of the Lyapunov-Schmidt reduction argument. However, we omit the details, as it is a  straightforward generalization of the method in \cite{AAW} (see also \cite{CEG}). We have,

\begin{theorem}
\label{main:thm}
Let $A = A_0$ and let $0< \beta + 1/p < 1/2$. Then, there exist $\epsilon_A>0$, $g_A>0$, $\kappa_A>0$, a unique continuous function 
$$
\kappa_* : B_{\epsilon_A}(0) \times B_{g_A}(0)  \rightarrow B_{\kappa_A}(0)  
$$
such that $\kappa_*(0,0) = 0$ and a unique continuous function 
$$
U_A : B_{\epsilon_A}(0) \times B_{g_A}(0) \times B_{\kappa_A}(0)  \rightarrow O({\al_*})\times \Wcal_{\beta(\rm even)}^{1,p}(\al_*) 
$$
such that $U_A(0,0,0) = (0, \om_A)$ and    
$$
G\big(U_A(\epsilon, g, \ka_*(\epsilon, g));\, \epsilon, g, \ka_*(\epsilon, g)\big) = 0.
$$
\end{theorem}

\begin{remark}
The implicit function theorem does not exclude the existence of traveling gravity-capillary waves close to the reference Crapper wave with negative upper-fluid density $\rho_1$ (corresponding to negative $\epsilon$). Those solutions are (presumably) unphysical.   
\end{remark}

\begin{section}{The construction of harmonic functions}\label{S.aux}

In the following two propositions, we solve Dirichlet problem on $\Om^c$ respectively on $\Om$ as required for the proof of Theorem \ref{thm3}. It will be more convenient to work with bounded domains, hence in this section $\Om$ is a bounded 'splash'-domain, symmetric w.r.t $x$-axis with the  splash point situated at $z_* = 0$, while $\Om^c = \Om_1 \cup \Om_2$ with $\Om_1$ bounded and situated in the right-half plane. The proof essentially follows the one given in \cite{MazSo}. In view of the future application to the dynamical case, we give quantitative bounds on the norms of solutions in terms of an appropriate weighted norm $|||z|||$ taking into account the failure of the arc-chord condition at the singular point. 

In both cases, for $\Om$ and $\Om_j$, the construction is reduced to solving the Dirichlet problem on the horizontal strip 
$$
\Pi := \{w:=\tau + i\nu \,:\, |\nu|< \pi/2\}
$$ 
by means of a suitable conformal transformation. In fact, let $\Phi_\pm$ be functions defined on the boundary $\pa\Pi_\pm :=\{\tau \pm i\}$ and consider the Dirichlet problem
$$
\begin{cases}
   \Delta \Phi = 0 \quad &\text{in $\Pi$}, \\
   \Phi = \Phi_\pm \quad &\text{on $\pa\Pi_\pm$}. 
\end{cases}
$$
It is not difficult to see, applying Fourier transform with respect to $\tau$ on the Laplace operator, that (modulo constants in front of each term) the Fourier transform of the solution to the above boundary value problem satisfies  
\be\label{eq:solF}
\widehat{\Phi}(\xi, \nu) \simeq (\h\Phi_+ + \h\Phi_-)(\xi) \frac{\cosh(\xi \nu)}{\cosh(\pi\xi/2)} + (\h\Phi_+ - \h\Phi_-)(\xi) \frac{\sinh(\xi \nu)}{\sinh(\pi\xi/2)},
\ee
with normal derivative
\be\label{eq:nsymF}
\aligned
\widehat{\pa_n\Phi}(\xi, \pm \pi/2) = \pm \widehat{\pa_\nu\Phi}(\xi, \pm \pi/2) &\simeq (\widehat{\Phi'}_+ + \widehat{\Phi'}_-)(\xi) \tanh(\pi\xi/2) \\
& - (\widehat{\Phi'_+} - \widehat{\Phi'_-})(\xi) \bigg(\coth (\pi\xi/2) - \frac{2}{\pi\xi}\bigg) + (\widehat{\Phi}_+ - \widehat{\Phi}_-)(\xi).
\endaligned
\ee
Note that $\coth $ is unbounded at zero and is therefore not a Fourier multiplier on weighted Lebesgue spaces with power type weights.

In the following lemma we obtain quantitative estimates on the conformal maps used in the proofs of both propositions below. Let $\al \rightarrow z(\al)$ for $\al \in [-\pi, \pi]$ be the arc-length parametrization of $\Ga$ and let $\varepsilon>0$ be fixed. Let furthermore $v$ be a bounded function on $[-\pi, \pi]$ such that $v(\al)=O(|\al \pm \al_*|)$ when $\al \in B_\varepsilon(\pm\al_*)$ and $v(\al) = O(1)$ otherwise. Then, we define 
$$
|||z||| :=  \|v(\al)^{1-\mu}\,\theta_\al\|_{H^1(\Tbb)} + \sup_{\al \in B_\varepsilon(\al_*)} \frac{1}{|v(\al)|^{1- \mu}|\theta_\al(\al)|} + \Fcal_{\varepsilon/2}(z) 
$$
where $\theta$ is the tangent angle and $\Fcal_{\epsilon/2}$ has been defined in \eqref{arcchord}. 

When $\de>0$ as in the introduction to section \ref{S.estimates}, it is not difficult to see that $\de \sim \varepsilon$ with the constant depending only on $|||z|||$.

%$z(\al)\in \Ga_{\de/2}:= \{z\in \Ga^+\cup \Ga^- \, :  \, |z_1| < \de/2 \, \}$ . 

\begin{lemma}\label{lem:r} Let $f: \Dbb\rightarrow \Om$ be a conformal map given by the Riemann mapping theorem where $f(\pm 1)=z_*$. Then    
 \be\label{est:rOm}
 |f'(\zeta)| \sim 1.
 \ee
 In particular, the conformal map $F:\Pi \rightarrow \Om$ defined by $F(w):= f\circ\tanh(w/2)$ satisfies
 \be\label{est:stripOm}
 \frac{|F'(w)|}{|1-\tanh^2(w/2)|} \sim 1, \qquad  \frac{|F(w)|}{|1-\tanh^2(w/2)|} \sim 1.
 \ee
 
Let now $f:\Dbb\rightarrow\Om_1$ be the conformal map given by the Riemann mapping theorem with $f(1)=z_*$. Then,    
\be\label{est:fOmc}
|f(\zeta)| \sim  \big|\log|1 - \zeta|\big|^{-1/\mu} \qquad \forall \zeta\in \overline{\Dbb}\cap B_{\t \de}(1)
\ee
for some $O(e^{-\exp(P|||z|||)}) = \t\de < 1$. Moreover, $|f|$ is bounded away from zero iff $|\zeta - 1|$ is bounded away from zero. On the other hand, the derivative satisfies 
\be\label{est:ffder}
|f'(\zeta)(1-\zeta)| \sim |f(\zeta)|^{\mu + 1} 
\ee
and therefore $|f'(\zeta)|\rightarrow \infty$ as $\zeta \rightarrow 1$. The corresponding conformal mapping $F:\Pi \rightarrow \Om_1$ satisfies  
\be\label{est:stripOmc}
|F(w)| \, \sim  \, |w|^{-1/\mu}, \qquad |F'(w)| \, \sim  \, |w|^{-1-1/\mu}, \qquad \forall w\in\overline{\Pi}: \Re w > |\log\t \de/4|
\ee
while outside of any neighborhood of $+\infty$, we have 
\be\label{est:stripOmFar}
|F(w)| \sim 1, \qquad \frac{|F'(w)|}{|1-\tanh^2(w/2)|} \sim 1 .
\ee

All the equivalence constants grow at most as a double exponential of $P(|||z|||)$ where $P$ is some, possibly different polynomial in each estimate.

\end{lemma}

Now we can state and prove both propositions:

\begin{proposition}
\label{prop:in}
Let $0<\beta + p^{-1}< \mu$ and let $\phi\in \Wcal^{1,p}_{\beta(even)}(\Ga)$. Then the Dirichlet problem 
$$
\begin{cases}
\Delta u = 0& \quad \mathrm{in } \,\,\Om^c, \\
 u = \phi& \quad \mathrm{on } \,\, \Ga,
\end{cases}
$$
has a solution $u$ harmonic at $\infty$, satisfying
\be\label{ineq:norm}
\|\pa_n u\|_{\Lcal_{p,\beta}(\Ga)}\lesssim \|\phi\|_{\Wcal^{1,p}_{\beta}(\Ga)}
\ee 
where the constant grows at most as a double exponential of $P(|||z|||)$ where is some polynomial.  
\end{proposition}

\begin{proof}
It is enough to consider the bounded component $\Om_1$ of $\Om^c$, the other can be transformed to a bounded domain via the conformal map $(z-x_1)^{-1}$ where $x_1\in \Om_1$. Let $f:\Dbb \longrightarrow \Om_1$ with $f(1) = z_*$ and $\t \de>0$ be as in Lemma \ref{lem:r}. Inspecting the proof of that same Lemma it is not difficult to see that we can choose $\de_1>0$ such that 
\be\label{cond:ngbh}
\Ga_1 \cap B_{\de_1}(z_*) \subseteq \Ga^+\cup \Ga^- \qquad f^{-1}(B_{\de_1}(z_*))\subseteq \pa\Dbb\cap B_{\t\de}(1).
\ee 
%For the second condition to be satisfied, it is not difficult to see  that $\de_1^\mu \gtrsim a / \t \de$. 
We first assume $\phi \equiv 0$ on $\Ga_1\cap B_{\de_1}(z_*)$. Then, 
$$
\aligned
&\int_{\pa\Dbb} |\phi\circ f(\zeta)|^p \, ds_\zeta = \int_{\Ga \setminus B(z_*)} |\phi(z)|^p |(f^{-1})'(z)| \, ds_z, \\
&\int_{\pa\Dbb} |\pa_{s_\zeta}(\phi\circ f)|^p \, ds_\zeta = \int_{\Ga \setminus B(z_*)} |\pa_{s_z}\phi(z)|^p |(f^{-1})'(z)|^{1-p} \, ds_z,
\endaligned
$$
hence Lemma \ref{lem:r} implies 
$$
 \phi\circ f \in W^{1,p}(\pa \Dbb).
$$ 
In particular, the normal derivative of the function $\Phi$ harmonic in $\Dbb$ with boundary value $\phi\circ f$ is given by 
$$
\pa_n \Phi(e^{is}) = \frac{1}{2\pi}\int \frac{d}{dt} \phi\big(f(e^{it})\big) \cot\bigg(\frac{s-t}{2}\bigg)\, dt.
$$
We set $u := \Phi \circ f^{-1}$. Estimates \eqref{est:fOmc}-\eqref{est:ffder} then imply
$$
\aligned
\|\pa_n u\|_{\Lcal_{p,\beta}(\Ga)} &= \Bigg(\int_{\pa \Dbb} |f(\zeta)|^{p\beta}|f'(\zeta)|^{1-p}\, |\pa_{n}\Phi(\zeta)|^p \, ds_\zeta\Bigg)^{1/p} \\
\, &\lesssim  \, \|\pa_n \Phi\|_{L^p(\pa \Dbb)} \, \lesssim  \, \| (\phi \circ f)'\|_{L^p(\pa \Dbb)}\,\lesssim \, \| \phi \|_{\Wcal^{1,p}_{\beta}(\Ga)}
\endaligned
$$
for any $\beta \in \Rbb$.

Let now $\phi\in\Wcal^{1, p}_{\beta(even)}(\Ga_1)$ be such that $\phi \equiv 0$ on $\Ga\setminus B_{\de_1}(z_*)$ and let $F:\Pi  \rightarrow \Om_1$ be the conformal map  $F(w) := f\circ \tanh(w/2)$. We may assume $\phi(0) = 0$; otherwise it only adds an overall constant to the required harmonic function. We set $\Phi_\pm(\tau) :=(\phi \circ F)(\tau \pm i\pi/2)$. It is not difficult to see that $F(w) \in \Ga_1\cap B_{\de_1}(z_*)$ implies $\Re w > \t \de $, cf. assumption \eqref{cond:ngbh}, hence asymptotic estimates \eqref{est:stripOmc} imply  
$$
\begin{alignedat}{2} 
\phi &\in \Lcal_{p, \beta-1}(\Ga_1) \quad &&\Leftrightarrow \quad \Phi_\pm \in \Lcal_{p,\, \mu^{-1} - \al}(\Rbb_+, w),\\ 
\phi' &\in \Lcal_{p, \beta}(\Ga_1) \quad &&\Leftrightarrow \quad \Phi_\pm' \in \Lcal_{p,\, 1 + \mu^{-1} - \al}(\Rbb_+, w),\\ 
\end{alignedat}
$$
where as usual we set $\al = \mu^{-1}(\beta + p^{-1}) + p^{-1}$ and the weight function is defined by $w(\tau) := (1 + \tau^2)^{1/2}$. 

As outlined in the beginning of this section, we use Fourier transform to solve Dirichlet problem on the strip. The harmonic extension $\Phi$ of $\phi\circ F$ is given by \eqref{eq:solF} and we first verify that 
$$
\pa_n\Phi \in \Lcal_{p, 1-\al + \mu^{-1}}(\Rbb,\, u)
$$ 
with the Fourier transform of $\pa_n\Phi$ given by the symmetric part of \eqref{eq:nsymF}. Note that if we wanted to prove the claim in the general case, we would additionally need 
$$
 \Phi_+ - \Phi_- \in \Lcal_{p,\, 1 + \mu^{-1} - \al}(\Rbb_+, w) \quad \Leftrightarrow \quad \phi_+ - \phi_- \in \Lcal_{p, \beta - \mu - 1}(\Ga^+\cup\Ga^-)
$$
(cf. \cite{MazSo}). Now the inverse Fourier transform of $\tanh (\pi\xi/2)$ is modulo a constant factor given by $(\sinh \tau)^{-1}$. It satisfies the assumptions of Lemma \ref{lem:aux1} (see Appendix) for any $J\geq 0$ and moreover, for a suitable $J \geq 0$ 
$$
0 < (1 - \al + \mu^{-1}) +  p^{-1} < 1 + \mu^{-1} \leq 1 + J
$$ 
so the claim follows from Lemma \ref{lem:aux1}. We set $u:=\Phi\circ F^{-1}$. It remains to show $\pa_n u\in \Lcal_{p,\beta}(\Ga_1)$. By construction we have $\pa_n u\in \Lcal_{p,\beta}(\Ga_1 \cap B_{2\de_1}(z_*))$. On the other hand, since $\phi'$ vanishes on $\Ga_1\setminus B_{\de_1}(z_*)$, it is not difficult to see that $\pa_n u = \frac{1}{F'\circ F^{-1}}\pa_n \Phi\circ F^{-1}$ is bounded on $\Ga_1\setminus B_{2\de_1}(z_*)$ and therefore $\pa_n u\in \Lcal_{p,\beta}(\Ga_1\setminus B_{2\de_1}(z_*))$.

\iffalse
In particular, we have the estimate
$$
\|\pa_n \Phi\|_{\Lcal_{p,1+\mu^{-1} - \al}(\pa \Pi)} \, \lesssim \, \|\Phi_{+}' + \Phi_{-}'\|_{\Lcal_{p,1+\mu^{-1} - \al}(\Rbb_+)} + \|\Phi_{+} - \Phi_{-}\|_{W^{1, p}_{1+\mu^{-1} - \al}(\Rbb_+)}
$$

Finally, let $\t \chi$ be a smooth function identically equal to $1$ on $\tau>1$ and vanishing on $\tau\leq 0$. We set 
$$
\Psi := \Delta (\t\chi (\Phi_{(+)} + \Phi_{(-)})). 
$$
Since (cf. \cite{MazSo})
$$
W^k_{p,\ga}(\Pi) \ni u \longrightarrow \big\{\Delta u, \, u|_{\pa \Pi} \big\} \in W^{k-2}_{p,\ga}(\Pi) \times W^{k-1/p}_{p,\ga}(\pa \Pi)
$$
performs an isomorphism for every $\ga\in \Rbb$ and $k \in \Nbb$, we know that $\Psi\in L^p(\Pi)$ ($k=1$). We now solve, 
$$
\begin{cases}
   \Delta Z = -\Psi \quad &\mathrm{in} \quad \Pi, \\
   Z = 0 \quad &\mathrm{on} \quad \pa \Pi, 
\end{cases}
$$
where the solution $U$ of the original problem is given by
$$
U := Z + \t\chi(\Phi_{(+)} + \Phi_{(-)}).
$$
We refer to \cite{MazSo}, for further details.
\fi
\end{proof}
\begin{proposition}
\label{prop:out}
Let $0<\beta + p^{-1}<1$ and let $\psi\in \Lcal_{p,\beta}(\Ga)$. Then the Neumann problem 
$$
\begin{cases}
\Delta v = 0& \quad \mathrm{in } \,\,\Om, \\
 \pa_n v = \psi& \quad \mathrm{on } \,\, \Ga,
\end{cases}
$$
has a solution $v$ satisfying 
$$
\|v\|_{\Wcal^{1,p}_{\beta}(\Ga)}\, \lesssim \,  \|\psi\|_{\Lcal_{p,\beta}(\Ga)}.
$$
and the constant grows at most as an exponential of some polynomial of $|||z|||$. 
\end{proposition}
\begin{proof}
By the Cauchy-Riemann equations, finding a solution of the above Neumann problem is equivalent to solving the corresponding Dirichlet problem with boundary value $\phi\in \Wcal^{1, p}_{\beta}(\Ga)$, where $\frac{d\phi}{ds} = \psi$ and $\phi(z_*) = 0$. 

Let $F:\Pi \rightarrow \Om$ be the conformal map from Lemma \ref{lem:r}. Then, estimates \eqref{est:stripOm} imply 
$$
\aligned
|F(w)|^{\beta - 1}|F'(w)|^{p^{-1}} &\sim \big|1 - \tanh^2 (w/2)\big|^{ - 1 + (\beta + p^{-1})}\\
|F(w)|^{\beta}|F'(w)|^{p^{-1} - 1} &\sim \big|1 - \tanh^2 (w/2)\big|^{ - 1 + (\beta + p^{-1})}
\endaligned
$$
and we have
$$
\big|1 - \tanh^2 (w/2)\big| \,\sim\, e^{-|\tau|}, \qquad w = \tau \pm i\pi/2.
$$ 
\iffalse
$$
\aligned
&\|\phi\|_{\Lcal_{p,\beta - 1}(\Ga)}^p = \int_{\pa \Pi } \underbrace{|F(w)|^{p(\beta - 1)}|F'(w)|}_{ \sim\big|1 - \tanh^2 \frac{w}{2}\big|^{p(\beta - 1) + 1}}\, |\phi\circ F(w)|^p  \, ds_w,\\
&\|\pa_s\phi\|_{\Lcal_{p,\beta}(\Ga)}^p =  \int_{\pa \Pi}  \underbrace{|F(w)|^{p\beta}|F'(w)|^{1-p}}_{\sim\big|1 - \tanh^2 \frac{w}{2}\big|^{p(\beta - 1) + 1}} \,|\pa_{s_w}(\phi\circ F)|^p  \, ds_w,
\endaligned
$$
\fi
In particular,
$$
\phi \in \Lcal_{p, \beta-1}(\Ga), \,\, \phi'\in\Lcal_{p, \beta}(\Ga) \quad \Leftrightarrow \quad \Phi_\pm, \, \, \Phi'_\pm \in \Lcal_{p, -1 + (\beta + p^{-1})}(\Rbb, \, e^{-|\tau|})
$$
where we have set $\Phi_\pm(\tau) = \phi\circ F(\tau \pm i\pi/2)$.

Let $\Phi$ denote the harmonic function on $\Pi$ with boundary value $\Phi_\pm$ on $\pa\Pi_\pm$, whose Fourier transform is given by \eqref{eq:solF} and that of its normal derivative by \eqref{eq:nsymF}. By Lemma \ref{lem:aux2} (see Appendix), both $\tanh(\pi\xi/2)$ and $\coth(\pi\xi/2) - \frac{2}{\pi\xi}$ are Fourier multipliers for exponential weights $e^{\pm d|\tau|}$ provided $d<1$, hence $\pa_n\Phi\in \Lcal_{p, - 1 + (\beta + p^{-1})}(\Rbb, \, e^{-|\tau|})$ and, in particular, $\pa_n(\Phi\circ F^{-1})\in \Lcal_{p,\beta}(\Ga)$. The rest is now straightforward.

\end{proof}

In both cases, we use a suitable conformal map $\psi$ in order to map $\Om$ to a regular domain $\t \Om$, that is, a bounded $\Ccal^{k,\lambda}$-domain where $0<\lambda<1$ and $k\geq 1$. In other words, a domain whose boundary $\t \Ga = \pa \t \Om$ admits a $\Ccal^{k,\lambda}$-parametrization $\al \rightarrow \t z(\al)$ for $\al \in [-\pi, \pi]$ satisfying the arc-chord condition, i.e. 
$$
\Fcal(\t z) := \sup_{\al, \beta}\Fcal(\t z)(\al, \beta) < \infty
$$
cf. \eqref{eq:defFcal} for the definition of $\Fcal(\t z)(\al, \beta)$. In particular, following \cite{CCG} we define the norm 
$$
|||\t z||| \sim \|\t z\|_{\Ccal^{k, \lambda}(\t\Ga)} + \Fcal(\t z).
$$ 
Then, by the Riemann mapping theorem, there exists a conformal map $\t f: \Dbb \longrightarrow \t \Om$, which can be extended as a $\Ccal^{k,\lambda}$-homeomorphism to $\overline{\Dbb}$  by the Kellog-Warschawski theorem. Let $\t z_0\in \t \Om$ be such that $\dist(\t z_0, \t\Ga)>> |||\t z|||^{-1}$ and let $\t f(0) = \t z_0$. Then there exists a polynomial $P$ such that  
\be\label{est:fder}
e^{-P(|||\t z|||)} \, \leq \, \frac{|\t f'(\zeta)|}{|\t f' (0)|} \, \leq \, e^{P(|||\t z|||)} 
\ee
(cf. \cite{CCG}). In particular, since both $\pa\Dbb$ and $\t \Ga$ are chord-arc curves, the above implies   
\be\label{est:f}
e^{-P(|||\t z|||)} \, \leq \, \frac{|\t f(\zeta_1) - \t f(\zeta_2)|}{|\zeta_1 - \zeta_2|} \, \lesssim \, e^{P(|||\t z|||)} \qquad \forall \zeta_1, \zeta_2\in \pa\Dbb
\ee
for some possibly different polynomial $P$. Indeed for any $\t z_1, \t z_2\in \pa\t\Om$ we have the following estimate
$$
\frac{\big|\t f^{-1}(\t z_1) - \t f^{-1}(\t z_2)\big|}{|\t z_1-\t z_2|} = \frac{1}{|\t z_1-\t z_2|}\bigg|\int_{\al_2}^{\al_1} (\t f^{-1}(\t z(\al)))'d\al \bigg|  \,\lesssim \, \frac{|\t z'|_\infty}{|\t f'|_\infty}\frac{|\al_1 - \al_2|}{|\t z(\al_1)-\t z(\al_2)|}, 
$$
where $\t z(\al_i) = \t z_i$. 

\begin{proof} (of lemma \ref{lem:r})

First, let $\t \Om$ be the image of $\Om$ under the complex square root. More precisely, for some $0<x_0 \in \Om_1$ to be specified we define $\psi(z):= \sqrt {z - x_0}$ and $\t \Om := \psi(\Om)$ where the branch cut is chosen along the negative real axis. Recall that $\Om^c = \Om_1 \cup \Om_2$ with $\Om_1$ bounded. More precisely, when $z\in \Ga^+ \cup \Ga^- $ is considered with respect to the graph parametrization (cf. introduction to section \ref{S.estimates}), we have 
$$
\psi(z) = \pm i e^{i\varphi(x)/2}\sqrt {|z - x_0|}, \quad \varphi(x) = \mp \arctan \bigg(\frac{\ka(x)}{x-x_0}\bigg), \qquad z = x \pm i \ka(x)\in \Ga^\pm
$$
and, in particular, $\psi(z_*) = \pm i\sqrt{x_0}$. We set $\t z(\al) := \psi(z(\al))$ and claim $\t z$ is a regular parametrization of $\t \Ga$ which has the same regularity as $z(\al)$ and satisfies the arc-chord condition. Moreover, $|||\t z|||$ is controlled by $P(|||z |||)$ for some polynomial $P$.

By assumption, for each $z\in \Ga\setminus \{z_*\}$ there exists $r(z)>0$ such that $B_{r}(z)\setminus \Ga$ has exactly two connected components while $B_{r}(z)\cap \Ga$ is connected. These $r(z)$ can be chosen uniformly on the complement of any neighborhood of $z_*$. In particular, setting $I_\epsilon := B_\varepsilon(\al_*)\cup B_\varepsilon(-\al_*)$, we may choose $x_0>0$ in such a way that 
$$
\dist(x_0, \Ga_1\setminus (\Ga^+\cup \Ga^-))\gtrsim r, \qquad r = \min_{\al \in [-\pi, \pi]\setminus I_\epsilon} r(z(\al)) = O(1/|||z|||^{k})
$$  
but also that $x_0 \sim \min\{r, 1\}$. Then obviously
$$
|\psi(z)| \, \gtrsim \,  \sqrt{x_0}, \qquad z\in  \Ga\setminus (\Ga^+ \cup \Ga^-),
$$
while a short calculation gives
$$
|\psi(z)| \, \gtrsim \,  \begin{cases}  \sqrt{x_0} & \quad |x-x_0|\geq x_0/2 \\
              \frac{1}{\sqrt{|||z|||}} \, x_0^{(\mu + 1)/2} & \quad |x-x_0| <  x_0/2
\end{cases}
$$
when $z\in \Ga^\pm$. In particular, for each $k=0, 1...$ there exists some polynomial $P$ such that
$$
|d^k\psi/dz^k| \, \sim \, 1, \qquad z\in \Ga, \qquad k=0,1...
$$
with the equivalence constant $O(e^{P(|||z|||)})$ and therefore $\t z$ has the same regularity as $z$ does. As for the arc-chord condition, given $\al\neq \beta$ such that $(\al, \beta)\in (B_\varepsilon(\al_*)\times B_\varepsilon(-\al_*))^c$ we have by definition of $\psi$ that 
$$
\Fcal(\t z)(\al, \beta) = \frac{|e^{i\al} - e^{i\beta}|}{|z(\al) - z (\beta)|} \, |\psi(z(\al)) + \psi(z(\beta))| \, \lesssim \, \sup_\al |\psi(z(\al))| \, \Fcal_{\epsilon}(z)
$$

\iffalse
$$
\Fcal(\t z)(\al, \beta) \,\leq \, 2 \sup_\al |\psi(z(\al))| \, \Fcal_{\epsilon}(z), \qquad (\al, \beta)\in (B_\varepsilon(\al_*)\times B_\varepsilon(-\al_*))^c
$$
\fi
On the other hand, let $(\al, \beta) \in B_\varepsilon(\al_*)\times B_\varepsilon(-\al_*)$. For convenience, let us use the graph parametrization, i.e. in the notation of section \ref{S.estimates} we have $z(\al)= z_+$ and $z(\beta) = q_-$. In particular,
$$
\frac{1}{|\psi(z_+) - \psi(q_-)|} = \frac{|\psi(z_+) + \psi(q_-)|}{|z_+ - q_-|} \, \leq \,  P(|||z|||) \, \underbrace{\frac{|\psi(z_+) + \psi(z_-)|}{|z_+ - z_-|}}_{:= t(x)}  + \frac{|\psi(z_-) - \psi(q_-)|}{|x - u|}  
$$
with the second term bounded by some power of $|||z|||$. It remains to consider $t(x)$. However, when $|x - x_0| \geq x_0/2$, the estimate $|\sin (\varphi(x)/2)| \,\lesssim \, |\ka(x)|/|x-x_0|$ implies 
$$
\aligned
t(x)   \, \lesssim \, \frac{\sqrt{|x- x_0| + |\ka(x)|}}{|\ka(x)|}|\sin(\varphi(x)/2)| \,\lesssim \, \frac{\sqrt{1 + |\ka(x)|/|x - x_0|}}{\sqrt{|x - x_0|}} \,\lesssim \, |||z|||^{1/2} (1 + |||z|||),
\endaligned
$$
while $|x-x_0|\leq x_0/2$, we have 
$$
t(x) \, \lesssim \, \frac{\sqrt{1 + |x- x_0|/|\ka(x)|}}{\sqrt{|\ka(x)|}} \, \lesssim \, |||z|||^{1/2} \, x_0^{-(\mu + 1)/2}(1 + x_0^{-\mu} |||z|||).
$$
In particular, there exists a polynomial $P(|||z|||)$ which majorizes both of the above and the arc-chord condition for $\t z$ follows.

By the Riemann mapping theorem there exists a mapping $\t f:\Dbb \rightarrow \t\Om$ satisfying estimates \eqref{est:fder}-\eqref{est:f} and we may assume $\t f(\pm 1) = \pm i \sqrt{x_0}$. We define $f:\Dbb\rightarrow \Om$ and $F: \Pi\rightarrow \Om$ to be the conformal mappings 
$$
f:= \psi^{-1}\circ \t f, \qquad F:= f\circ \tanh(w/2).
$$
Note that $F$ maps $\pa\Pi^\pm$ to $\pa \Om_1$ respectively $\pa\Om_2$. Taking into account that $\psi^{-1}(\t z) = \t z^2 + x_0$, we have
$$
|f'(\zeta)| = 2|\t f(\zeta)||\t f'(\zeta)|, \qquad  \frac{|F'(w)|}{|1 - \tanh^2(w/2)|}= \frac{1}{2}\big|f'(\zeta)\big|  
$$
and therefore $|f'|$ and $|F'|$ satisfy estimates \eqref{est:rOm} and \eqref{est:stripOm} respectively (using $|||\t z||| \lesssim P(|||z|||)$ and  $|\t f(\zeta)| = |\psi(z)|$). 

It remains to show estimate \eqref{est:stripOm} for $F$. The upper bound is straightforward, once we have the estimate for $f'$. As for the lower bound, let $w\in \pa\Pi$ be such that $z=F(w)\in \Ga\setminus (\Ga^+\cup\Ga^-)$. Then $f$ is invertible on e.g. $f^{-1}(\Ga\setminus \Ga^-)$ and assuming $\Re(w)\geq 0$ we have the estimate
$$
\aligned
 \frac{\big|1 - \tanh^2 \frac{w}{2}\big|}{|F(w)|} = \frac{\big|1 - (f^{-1}(z))^2\big|}{|z|} \, \lesssim \, \frac{1}{|z(\al)|}\int_{\al_*}^\al \frac{|z'(\al')|}{\big|f'\circ f^{-1}( z(\al'))\big|}\, d\al' \, \lesssim \, \Fcal_{\varepsilon/2}(\Ga) \, e^{P(||||\t z|||)}
\endaligned
$$
and similarly if $\Re(w)<0$. If, however $z\in \Ga^\pm$ then use the bijection $f: f^{-1}(\Ga^\pm) \rightarrow \Ga^\pm$ and since $\Ga^\pm$ can be parametrized as a graph, we replace $\Fcal_{\varepsilon/2}$ with $(\sup |z_\al|)^{-1}$.

Let us now consider the 'cusp' domain $\Om^c$. W.l.o.g. we restrict attention to the bounded component $\Om_1$ of $\Om^c$. The unbounded component $\Om_2$ can be reduced to the bounded case via the transformation $1/(z - x_1)$ for some $x_1\in \Om_1$. %with $x_1 >> 1/|||z|||$.
Let $\al \rightarrow z(\al)$ with $\al \in [-\pi, \pi]$ be a regular parametrization of $\Ga$ with the singular point located at $\al = 0$. We first map $\Om_1$ to a curvilinear half-strip via the conformal map $z^{-\mu}$. Near the singular point, we have 
$$
z^{-\mu} = e^{\mp i \mu \varphi(x)} |z|^{-\mu}, \quad \varphi(x) = \arctan\Big(\frac{\ka(x)}{x}\Big), \quad z\in \Ga^\pm 
$$
and it is not difficult to see it approaches the lines $\tau \mp ik/(\mu + 1)$ asymptotically as $\tau\rightarrow \infty$, with $k$ as in \eqref{reg_implication}. Assume for simplicity that $(-az^{-\mu})(\Om_1)\subseteq\{(\tau, \nu):\, \, |\nu|<\pi\}$ where $a:= \pi(\mu + 1)/2k$. Then  
$$
\psi(z) = \exp(-az^{-\mu})
$$
maps the interior of $\Om_1$ to a bounded domain $\t \Om$ with $\Ccal^{1,\lambda}$-boundary $\t \Ga$. Note that $\t z (\al) := \psi(z(\al))$ in not a regular parametrization in the neighborhood of the singular point since $|\psi'(0)| = 0$. However, the tangent vector at $\psi(z_\pm(x))$ reads
$$
\frac{(\psi\circ z_\pm)'(x)}{|(\psi\circ z_\pm)'(x)|} = e^{\pm i \t \theta(x)}, \qquad \t\theta(x)= \theta(x) + (\mu + 1)\varphi(x) + a |z|^{-\mu}\sin (\mu \varphi(x))
$$
where $\frac{z_\pm'(x)}{|z'_\pm(x)|} = e^{\pm i \theta(x)}$. By construction, we have $\lim_{x\rightarrow 0}\t\theta(x) = \pi/2$ and $\t \theta\in H^1$ with the norm controlled by some polynomial of $|||z|||$. Note that the derivative of the last term of $\t\theta(x)$ is to the first order controlled by $(x^{-(\mu + 1)}\ka(x))'$ which belongs to $L^{2}$. %In case $\mu = 1$ and under the assumption $\pa_x^3\ka\in L^p$ for some $p$ this belongs to some $L^{p'}$.
Setting
$$
\t z_\pm (x) := \int^x_0 e^{i\t \theta(s)}ds
$$
we obtain an arc-length parametrization on the neighborhood of $0$. It remains to verify the arc-chord condition. However, by making $\de$ smaller we can ensure $|\t \theta(x) - \frac{\pi}{2}| \lesssim P(|||z|||)x^\mu \leq \frac{\pi}{4}$ hence
$$
|\t z(x) - \t z(-u)| \geq \bigg|\int_0^x \sin(\t\theta(s))ds + \int_0^u \sin(\t\theta(s))ds\bigg| \gtrsim |x + u|
$$
We clearly have $|d^k\psi(z)/dz^k|\lesssim 1$ for all $z\in\Ga$ and $k=0, 1..$. However, we only have lower bounds outside of some neighborhood of the origin which moreover, depend exponentially on the size of the neighborhood. In particular, it is not difficult to verify the remaining part of the arc-chord condition, however, piecing everything together we only have $|||\t z|||\lesssim e^{P(|||z|||)}$.

Let us now define the corresponding mappings $f(\zeta) := \psi^{-1} \circ \t f(\zeta)$ and $F(w):= f\circ \tanh(w/2)$ where $\t f:\Dbb\rightarrow \t \Om$ is given by the Riemann mapping theorem and satisfies estimates \eqref{est:fder} and \eqref{est:f}. We assume in addition $\t f(1) = \psi(z_*) = 0$. In the interior of $\Dbb$, we have by definition
\be\label{eq:derfcusp}
\aligned
a f(\zeta)^{-\mu} = -\log \t f(\zeta), \qquad f'(\zeta) = \frac{1}{\mu a }\frac{\t f'(\zeta)}{\t f(\zeta)} f(\zeta)^{\mu + 1} %= -\frac{a^{1/\mu}}{\mu}\frac{\t f'(\zeta)}{\t f(\zeta)} \big(\log\t f(\zeta)\big)^{-(1 + \mu^{-1} )}
%&(f^{-1})'(z) = -\mu a \, \frac{\psi(z)z^{-\mu - 1}}{\t f' ( f^{-1}(z))}
\endaligned
\ee
and $f'$ can be extended to the boundary of $\Dbb$ in any Stolz angle with vertex on $\pa\Dbb\setminus \{1\}$. In particular, estimates \eqref{est:fder}-\eqref{est:f} imply estimate \eqref{est:ffder}, i.e. we have
$$
|(1 - \zeta) f'(\zeta)|\,\sim \,\frac{1}{a} |f(\zeta)|^{\mu + 1}. 
$$
Let $\t \de := e^{-2(P(|||\t z|||) + \pi/2)}$ with $P$ as in estimate \eqref{est:f} and let $|\zeta - 1| < \t \de$. %By possibly making $\t \de$ smaller, we may assume $\t \de < 1$. 
Estimate \eqref{est:f} can be rewritten as 
\be\label{eq:estfaux}
\big|\log |\t f (\zeta)| - \log |\zeta - 1|\big|\, \leq\, P(|||\t z|||) 
\ee
hence, by the choice of neighborhood, we have $\log |\t f(\zeta)| < 0$ and  
$$
\frac{1}{2}|\log|\zeta - 1|| \, \leq \, |\log |\t f(\zeta)||\,\leq\, |\log \t f(\zeta)| \,\leq\, |\log|\t f(\zeta)|| + \pi/2 \, \leq \, 2|\log|\zeta - 1||.
$$
In particular \eqref{est:fOmc} follows, i.e. we have
$$
a^{-1/\mu}|f(\zeta)|\sim |\log|\zeta - 1||^{-1/\mu} \qquad \forall \zeta \in \overline{\Dbb} : |\zeta - 1|<\t \de.
$$ 
By going over to the strip $\Pi$ we recover power growth at infinity. Indeed let $\zeta = \tanh(w/2)$ with $\Re w \geq \log 2$. Then
$$
|\zeta - 1| = \frac{2 e^{-\Re w}}{|1 + e^{-w}|}, \qquad \big|\big|1 + e^{-w}\big| - 1\big| \leq e^{-\Re w} \leq 1/2
$$
and therefore
$$
e^{-\Re w} < \t \de/4 \quad \Rightarrow \quad |\zeta - 1| < \t \de. %\quad \Rightarrow \quad e^{-\Re w}} < \t \de 
$$
In particular, estimate \eqref{est:fOmc} implies
$$
|F(w)| \sim \big|\log \big|1 - \tanh(w/2)\big|\big|^{-1/\mu} \, \sim \, |w|^{-1/\mu}, \qquad \forall w : \Re w \geq |\log \t \de/4|.
$$
Finally, equation \eqref{eq:derfcusp} implies  
$$
|F(w)|^{1 + \mu} \sim |F'(w)| \quad \Rightarrow \quad |F'(w)| \sim |w|^{-1-1/\mu}.
$$
  
Finally, we show that $|f|$ is bounded away from zero iff $|1 - \zeta|$ is. In fact, combining the first equation in \eqref{eq:derfcusp} with \eqref{eq:estfaux}, we have   
$$
|f(\zeta)|\, \gtrsim\, \bigg(\frac{a}{ |\log |\t f(\zeta)|| + \pi/2}\bigg)^{1/\mu} \, \gtrsim\, \bigg(\frac{a}{ |\log |1 - \zeta|| + P(|||\t z|||) + \pi/2}\bigg)^{1/\mu}.
$$
and conversely 
$$
|\log|1 - \zeta|| \, \lesssim \, a |f(\zeta)|^{-\mu} + P(|||\t z|||).
$$
As for the upper bound, note that we have $|f(\zeta)|\leq M:= \max\{1,\, \sup_\al |z(\al)|\}$, since $\Om_1$ is bounded. In particular, the corresponding estimates for $F$ and $F'$ follow. 
\end{proof}

\end{section}
\section*{Acknowledgements}

A.E.\ and N.G.\ are supported by the ERC Starting Grant~633152. This
work is supported in part by the Spanish Ministry of Economy under the
ICMAT--Severo Ochoa grant SEV-2015-0554 and the MTM2017-89976-P. 
788250. D.C.\ was partially supported by the ERC Advanced Grant~788250.

\appendix

\section{Appendix}\label{S.appendix}

\begin{subsection}{The variable change $\tau\rightarrow h(\tau)$}
\iffalse
W.l.o.g. let $x\in (0,\de)$ be fixed and let $q = u \pm i\ka(u)\in \Ga^\pm$. Following \cite{MazSo}, we divide the right-hand side $\Ga^\pm\cap\{\Re z>0\}$ into three parts
$$
\aligned
&\Ga_l^\pm(x) := \{q\in \Ga^\pm\,:\, x-u<\varep x\}, \\
&\Ga_c^\pm(x) := \{q\in \Ga^\pm\,:\, |x-u|\leq\varep x\}, \\
&\Ga_r^\pm(x) := \{q\in \Ga^\pm\,:\, u-x>\varep x\}, \\
\endaligned
$$
for some $1>\varep>0$. We will mostly just work with the right-hand side of $\Ga^\pm$, since given $x>0$, the left-hand side $\Ga^\pm\cap\{\Re z<0\}$ can be divided into two parts which can be estimated as $\Ga_l^\pm$ respectively $\Ga_r^\pm$. 
\fi
Let $\rho(\nu) := \ka_+(\nu) - \ka_-(\nu)$, where, by making $\de$ smaller if necessary, we assume $\rho$ is strictly monotonically increasing on $[0,\de]$. We implicitly define 
$$
h^{-1}(u) := \int^\de_{u} \frac{d\nu}{\rho(\nu)}, \qquad u\in(0,\de)
$$
(cf. \cite{MazSo}). Then $h^{-1}:(0,\de) \rightarrow (0, \infty)$ is strictly monotonically decreasing and it is three times continuously differentiable ($\rho$ is at least two times continuously differentiable on $(0,\de)$). In particular, the inverse $h$ exists and satisfies 
\be\label{eq:hder}
h'(\tau) = -\rho(h(\tau)).
\ee
Regularity assumptions \eqref{reg_assum} imply (after possibly making $\de$ smaller) that 
$$
\frac{1}{\rho(\nu)} \sim \frac{1}{k\nu^{\mu + 1}}
$$
and therefore   
$$
h^{-1}(u) \sim \frac{\mu}{k}\big(u^{-\mu} - \de^{-\mu} \big).
$$
In particular, 
$$
h(\tau) \sim (1 + \tau)^{-\frac{1}{\mu}}
$$
with constants depending only on $k,\de$ and $\mu$. By taking derivatives of formula \eqref{eq:hder}, we see the asymptotic formula can be differentiated three times, i.e. 
$$
|h^{(k)}(\tau)|\sim (1 + \tau)^{-k - \frac{1}{\mu}}, \qquad k=1,2,3.
$$

\end{subsection}

\subsection{Hardy operator and compactness}

We will need the following lemma on continuity and compactness of Hardy operator in weighted Lebesgue spaces (cf. \cite{KufOpic}):
\begin{lemma}[Hardy's inequality]\label{lem:hardy}
Let $p>1$ and let $R < \infty$. 
\begin{enumerate}
 \item\label{itm:hardy1} Let $\beta + p^{-1} < 1$, then  
$$
f\longrightarrow \int_0^x f(t)dt \,:\, \Lcal_{p,\beta}([0,R]) \longrightarrow \Lcal_{p,\ga}([0,R]) 
$$
is continuous for all $\ga \geq \beta -1$ and compact for all $\ga > \beta -1$. 
\item\label{itm:hardy2} Let $\lambda + p^{-1} >0$, then 
$$
f\longrightarrow \int_x^R f(t)dt \,:\, \Lcal_{p,\ga}([0,R]) \longrightarrow \Lcal_{p,\beta}([0,R])
$$
is continuous for all $\ga \leq \beta + 1$ and compact for all $\ga < \beta + 1$.

\end{enumerate}
\end{lemma}

In order to show compactness for certain singular integral operators, we will also need the following result:

\begin{lemma}
\label{lem:compact}
Let $a:[0,\de]\rightarrow \Rbb$ be continuous with $a(0) = 0$ and let 
$$
A\om(x):=\int_0^\de\om(u)\frac{1}{(x-u) + i\rho(u)}\,du; \qquad A_*\om(x):=\int_0^\de\om(u)\frac{1}{(x-u) + i\rho(x)}\,du. 
$$
Then $A$ and $A_*$ are continuous as operators $\Lcal_{p,\beta}\longrightarrow\Lcal_{p,\beta}$. Moreover $aA,\,Aa,\,aA_*,\,A_*a$ and their complex conjugates are compact.
\end{lemma}
\begin{proof}
We only need to show compactness, continuity follows as in Proposition \ref{prop:BR}. For $k\in\Nbb$ let $\chi_k:[0,\de]\rightarrow [0,1]$ be a smooth cut-off function such that $\chi_k(u) = 1$ if $u<\frac{\de}{k}$ and $\chi_k(u) = 0$ if $u\geq \frac{2\de}{k}$. We set
 $$
 a_k:=(1-\chi_k)a;\quad B:=Aa.
 $$
 The operators $B_k\om(x) :=A(a_k\si)(x)$ are compact since $\frac{a_k(u)}{(x-u) + i\rho(u)}$ are bounded kernels ($a_k$ are identically vanishing near $0$). It is not difficult to see that  
 $$
 \sup_{\|\si\| = 1}\|(B_k-B)\si\|_{p,\beta}\longrightarrow 0
 $$
 and therefore $B$ is compact as the limit of a sequence of compact operators. Others follow similarly.
 \iffalse
 Indeed, let $\|\si\|_{p,\beta} = 1$. Then, by continuity of $A$, we have
 $$
 \|(B_k-B)\si\|_{p,\beta} = \|A(\chi_k a \si)\|_{p,\beta}\leq \|A\|\,|\chi_ka|_{\infty}.
 $$
 Let $\varep>0$, then  
 $$
% \aligned
 \exists \de'>0:\,|a(u)|<\varep,\,\,\forall u<\de' \quad \Rightarrow \quad \exists k_0\in\Nbb: \, |\chi_k a|<\varep,\,\, \forall k\geq k_0
% \endaligned
 $$
 by continuity of $a$ and therefore
 $$
 \sup_{\|\si\| = 1}\|(B_k-B)\si\|_{p,\beta} \leq c\varep, \quad \forall k\geq k_0.
 $$
 \fi
\end{proof}

\subsection{Fourier multipliers on weighted Sobolev spaces}

We now state few important lemmas related to the Fourier multiplier
theorems on certain weighted Lebesgue spaces. The first, proof of
which can be found in \cite{MazSo}, gives continuity properties of an
integral operator with prescribed decay at infinity. More precisely,
let us define
$$
\phi\in \Lcal_{p, \ga}(\Rbb) \, \Leftrightarrow \,
(1+\tau^2)^{\ga/2}\phi \in L^p(\Rbb)\,;
$$
then one has:

\begin{lemma} 
\label{lem:aux1}
Let $T$ be an integral operator on $\Rbb$ with kernel $K(x,y)$, satisfying, for some $J\geq 0$ the estimate
$$
|K(x,y)| \lesssim \frac{1}{|x-y|}\frac{1}{(1 + |x-y|^J)}.
$$ 
If $0<\ga + p^{-1}<1 + J$ and $T:L^p(\Rbb)\longrightarrow L^p(\Rbb)$ is continuous, then 
$$
T: \Lcal_{p, \ga}(\Rbb) \longrightarrow \Lcal_{p, \ga}(\Rbb)
$$
is continuous. 
\end{lemma}

In the construction of harmonic functions on $\Om$ we work with weighted Lebesgue spaces with exponential weights. We make use of the following Fourier multiplier result (cf. \cite{Schott}): 
\begin{theorem}
Let $1<p<\infty$ and $u(x)=\exp(\pm d|x|)$ with $d>0$. Assume there exists a constant $c>0$ such that
\begin{equation}
\label{eq:multiplier}
\sup_{\xi\in\Rbb} \langle\xi \rangle^{|\beta + \ga|} |D^{\beta + \ga}_\xi a(\xi)| \leq c \frac{\beta^\beta}{(ed)^{|\beta|}}; \qquad \forall\xi \in \Rbb^n,\, \beta\in\Nbb^n, \, \ga\leq 2n +2
\end{equation}
where $\langle\xi\rangle:=\sqrt{1 + \xi^2}$. Then, linear operator $T_a$, defined via 
$$
\widehat{T_a \phi}(\xi):= a(\xi)\h\phi(\xi) 
$$
on $L^2\cap L^p$ is bounded as an operator $\Lcal^p(u)\rightarrow \Lcal^p(u)$.
\end{theorem}
We need to verify that \eqref{eq:multiplier} is satisfied for:
\begin{lemma}
\label{lem:aux2} 
Let $d<1$ and $n=1$. Then
\begin{enumerate}
 \item $a_1(\xi):= \tanh(\frac{\pi\xi}{2})$
 \item $a_2(\xi):= \coth(\frac{\pi\xi}{2}) - \frac{2}{\pi\xi}$
\end{enumerate}
satisfy \eqref{eq:multiplier}.
\end{lemma}
\begin{proof}
For $a_1$, note that hyperbolic tangent can be written as a fractional series
$$
\aligned
\tanh\Big(\frac{\pi\xi}{2}\Big) &= \frac{4\xi}{\pi} \sum_{k=1}^\infty \frac{1}{(2k-1)^2 + \xi^2}  \\
&=\frac{2}{\pi} \sum_{k=1}^\infty \bigg[\frac{1}{\xi + i\xi_k} + \frac{1}{\xi - i\xi_k}\bigg], \qquad \xi_k:=2k-1 \\
&=\frac{4}{\pi} \sum_{k=1}^\infty \frac{\cos\theta_k(\xi)}{\rho_k(\xi)}; \qquad \rho_k(\xi):=\sqrt{\xi_k^2 + \xi^2},\,\,\tan\theta_k(\xi)=\frac{\xi_k}{\xi}
\endaligned 
$$
In particular, for $n\in\Nbb_0$, we have
$$
\aligned
D^n_\xi \tanh\Big(\frac{\pi\xi}{2}\Big) &= (-1)^n n! \frac{2}{\pi} \sum_{k=1}^\infty \Bigg[\frac{1}{(\xi + i\xi_k)^{n+1}} + \frac{1}{(\xi - i\xi_k)^{n+1}}\Bigg] \\
&=(-1)^n n! \frac{2}{\pi} \sum_{k=1}^\infty \frac{\cos\big((n+1)\theta_k(\xi)\big)}{\rho^{n+1}_k(\xi)}
\endaligned
$$
and
$$
\aligned
\frac{\pi}{2}\langle\xi\rangle^n\Big|D^n_\xi &\tanh\big((\pi\xi)/2\big)\Big| \leq n! \sum_{k=1}^\infty \frac{\rho_1(\xi)}{\rho_k(\xi)^2} \\
&\leq n! \rho_1(\xi)\,\int_1^\infty \frac{dy}{ (2y-1)^2 + \xi^2} 
\endaligned
$$
where the right-hand side 
$$
f(\xi) := \rho_1(\xi)\,\int_1^\infty \frac{dy}{ (2y-1)^2 + \xi^2} = \frac{\rho_1(\xi)}{2\xi}\bigg(\frac{\pi}{2}- \arctan\frac{1}{\xi}\bigg)
$$
is monotonically increasing for $\xi\geq 0$ with $\lim_{\xi\rightarrow \infty}|f(\xi)|=\frac{\pi}{4}$, so in particular
$$
\sup_{\xi \in\Rbb}\langle\xi\rangle^n\big|D^n_\xi \tanh\big((\pi\xi)/2\big)\big|\leq \frac{n!}{2}.
$$
For $\ga\leq 4$, we have
$$
\sup_{\xi \in\Rbb}\langle\xi\rangle^{n+\ga}\big|D^{n+\ga}_\xi \tanh\big((\pi\xi)/2\big)\big| \leq \frac{n!}{2d^n}P(n)d^n
$$
where $P(n) = (n+4)!/n!$ is a polynomial of order 4 in $n$. The Stirling formula implies 
$$
n! \sim \sqrt{2\pi n}\Big(\frac{n}{e}\Big)^n,
$$
so we are finished if we can show $\sqrt{n} P(n)d^n$ is bounded, but this is true provided $d<1$.

Similarly, for $a_2$ we can write
$$
a_2(\xi)=\coth(\frac{\pi\xi}{2}) - \frac{2}{\pi\xi} = \frac{4\xi}{\pi} \sum_{k=1}^\infty \bigg[\frac{1}{\xi + i\xi_k} + \frac{1}{\xi - i\xi_k}\bigg]; \qquad \xi_k = 2k, 
$$
and the proof follows analogously.
\end{proof}

\end{document}